\documentclass[reqno,11pt]{amsart}
\usepackage{latexsym,amsmath,amssymb,amsbsy,amscd,bm,amsfonts}
\allowdisplaybreaks
\usepackage{amsmath,amssymb}
\usepackage{marginnote}
\usepackage[active]{srcltx}
\usepackage{amsmath,amsfonts,latexsym,graphicx,amssymb,url}
\usepackage{hyperref}
\hypersetup{
  colorlinks   = true, 
  urlcolor     = blue, 
  linkcolor    = blue, 
  citecolor    = red   
}
\theoremstyle{plain}
\newtheorem{theorem}{Theorem}[section]

\newtheorem{lemma}{Lemma}[section]

\newtheorem{cor}{Corollary}[section]

\theoremstyle{definition}
\newtheorem{definition}{Definition}[section]

\theoremstyle{remark}

\newtheorem{remark}{Remark}[section]

\numberwithin{table}{section}
\numberwithin{equation}{section}
\newcommand{\real}{{\mathbb R}}
\newcommand{\roots}{{\mathcal R}}
\newcommand{\proots}{{\mathcal R}^+}
\newcommand{\comp}{{\mathbb C}}

\newcommand{\G}{{\mathbb G}}
\newcommand{\Lie}{{\mathcal G}}
\newcommand{\T}{{\mathcal T}}
\newcommand{\To}{{\mathbb T}}
\newcommand{\Ho}{{\mathcal H}}

\newcommand{\eps}{\varepsilon}
\newcommand{\xe}{X^{\eps}}
\newcommand{\ue}{u_{\varepsilon}}
\newcommand{\ve}{\nabla_{\hspace{-0.1cm}\T}^{\eps}}
\newcommand{\vv}{\nabla_{\hspace{-0.1cm}\T}}
\newcommand{\he}{\nabla_{\hspace{-0.1cm}\Ho}}
\newcommand{\we}{\omega_{\varepsilon}}
\newcommand{\nae}{\nabla^{\eps}}
\DeclareMathOperator{\loc}{loc}
 \DeclareMathOperator{\supp}{supp}
 \DeclareMathOperator{\Ad}{Ad}
 \DeclareMathOperator{\ad}{ad}

 \DeclareMathOperator{\spa}{span}
 \DeclareMathOperator{\tr}{trace}

  \DeclareMathOperator{\SU}{SU}
   
     \DeclareMathOperator{\su}{su}
  \DeclareMathOperator{\GL}{GL}
 \DeclareMathOperator{\gl}{gl}

\begin{document}
\title[$C^{1,\alpha}$-subelliptic regularity]{$C^{1,\alpha}$-subelliptic regularity on $\SU (3)$ and compact, semi-simple Lie groups}
\author{Andr\'{a}s Domokos}
\address{Department of Mathematics and Statistics,
California State University Sacramento, 6000 J Street, Sacramento, CA, 95819, USA}
\email{domokos@csus.edu}
\author{Juan J. Manfredi}
\address{Department of Mathematics, University of Pittsburgh, Pittsburgh, PA, 15260}
\email{manfredi@pitt.edu}

\date{\today}



\begin{abstract} Let the vector fields $X_1, ... , X_{6}$ form an orthonormal basis of $\mathcal{H}$, the orthogonal complement of a Cartan subalgebra (of dimension $2$) in $\SU(3)$. We prove that weak solutions $u$  to  the degenerate   subelliptic $p$-Laplacian
$$ \Delta_{\mathcal{H},{p}} u(x)=\sum_{i=1}^{6} X_i^{*}\left(|\he u|^{p-2}X_{i}u \right) =0,$$
have H\"older continuous horizontal derivatives $\he u=(X_1u, \ldots, X_{6}u)$  for $p\ge 2$. 
\par
We also prove that a similar result holds for all compact connected semisimple Lie groups. 
\end{abstract}

\maketitle

\section{Introduction}\label{sec:1Intro}
Given a set of $m$ vector fields $X_1, X_2, \ldots,X_m$, in a domain $\Omega\subset\mathbb{R}^N$, where $m\le N$,   the horizontal gradient of a function $u\colon\Omega\mapsto\mathbb{R}$ is the vector field
$$\he
  u= X_1(u)X_{1}+X_2(u) X_{2}+ \ldots+X_m(u)X_{m}.$$
For $p\ge 1$ the horizontal Sobolev space $W_{\mathcal{H}}^{1,p}(\Omega)$ consists of functions $u$ for which we have
$$\|u\|_{W_{\mathcal{H}}^{1,p}(\Omega)}= \left(\int_{\Omega} \left(|\nabla_{\mathcal{H}}u(x)|^{p}+|u(x)|^{p}\, \right) dx \right)^{1/p}< \infty.$$
Here we have used $$ |\he u|= \left(\sum_{i=1}^{m} (X_{i}u)^{2}\right)^{1/2}.$$
As usual, we define  $W_{\mathcal{H}, 0}^{1,p}(\Omega)$ as the closure in  the $W_{\mathcal{H}}^{1,p}(\Omega)$-norm of the the smooth functions with compact support.  Given a function $F\in W_{\mathcal{H}}^{1,p}(\Omega)$, consider the variational problem
\begin{equation}\label{pvariational}
\inf_{u-F\in W_{\mathcal{H}, 0}^{1,p}(\Omega) }\int_{\Omega} |\nabla_{\mathcal{H}}u(x)|^{p}\, dx.
\end{equation}

When $p>1$ there exists a minimizer, that it is also unique when the vector fields satisfty the 
 H\"ormander condition
\begin{equation}\label{horcondition}
\text{rank Lie span}\{X_1, X_2, \ldots,X_m \}(x)=N\text{ for all } x\in \Omega,
\end{equation}
which we assume from now on.  Minimizers of \eqref{pvariational} are weak solutions of the \textit{subelliptic} or \textit{horizontal} $p$-Laplacian
\begin{equation}\label{plapequation}
 \Delta_{\mathcal{H},{p}} u(x)=\sum_{i=1}^m X_i^{*}\left(|\he u|^{p-2}X_{i}u \right) =0,
\end{equation}
where $X_{i}^{*}$ is the adjoint of $X_{i}$ with respect to the Lebesgue measure. Note that in the linear case $p=2$ we get $$ \Delta_{\mathcal{H},{2}} u(x)= \sum_{i=1}^m X_i^* X_{i}u (x).$$
If the dimension of the Lie algebra generated by $X_1, X_2, \ldots,X_m$ at each point $x$ is $N$ (H\"ormander's condition \eqref{horcondition}), then it is well-known that the operator $\Delta_{\mathcal{H},{2}}$ is hypoelliptic \cite{H67}. In fact, H\"ormander proved several estimates in $L^2$-fractional Sobolev spaces. These estimates were extended to more general $L^p$-fractional Sobolev and Besov space by Rothschild and Stein \cite{RS76}.  \par
In the quasilinear case $p\not =2$, when the non-degeneracy  and boundedness condition for the horizontal gradient
\begin{equation}\label{nondegenerateboundedness}	
0<\frac{1}{M} \le |\he u|(x) < M, \text{ for a. e. } x\in\Omega.
\end{equation}
is satisfied, Capogna \cite{C97,C99} proved that solutions to \eqref{plapequation} are $C^{\infty}$-smooth for the Heisenberg group, and Carnot groups, respectively.  The case of general semi-simple Lie groups follows from work done by us in \cite{DM09} for special classes of vector fields. 
\par
The situation is more complicated when we only assume  the non-degeneracy condition for the horizontal gradient
\begin{equation}\label{nondegenerate}	
0<\frac{1}{M} \le |\he u|(x), \text{ for a. e. } x\in\Omega \, .
\end{equation}
In this case the key step is to show first the boundedness of the horizontal gradient. In the case of the Heisenberg group this is due to Zhong \cite{Z17}, who extended the Hilbert-Haar theory to the Heisenberg group. Assuming \eqref{nondegenerate},  Ricciotti \cite{R15} proved $C^{\infty}$-smoothness of $p$-harmonic functions in the Heisenberg group for $1<p<\infty$. 
This result was extended to general  contact structures by using Riemannian approximations in  \cite{CCDO18}, which is the method we will extend below.\par
When condition \eqref{nondegenerate} is not assumed, we can only expect $C^{1,\alpha}$-regularity as in the Euclidean case. For the Heisenberg group this is indeed the case. See \cite{R18} for the case $p>4$, \cite{Z17} for $p>2$, and \cite{ZM17} for $1<p<\infty.$\par
The case of  general  contact structures is considered in  \cite{CCDO18}, where the $C^{1, \alpha}$-regularity of $p$-harmonic functions is obtained for $p\ge 2$.\par

In this paper, we consider first  the group $\SU(3)$ and second,  all compact, connected, semi-simple Lie groups,  and prove  that if $u$ is a solution of \eqref{plapequation} and $p\ge 2$,  then $\he u$ is H\"older continuous. As we shall explain below,  the dimension of the space of non-horizontal vectors fields, which turns out to be the dimension of the maximal torus, may be  greater than 1; thus,  it cannot support  a contact structure since  the dimension of the non-horizontal subspace is greater  than or equal to two. \par

We extend the Riemannian approximation method of \cite{CC16} to $\SU(3)$ (and general semisimple compact Lie groups) to get boundedness of the gradient, and build on the work of
\cite{MM07},\cite{ DM09}, \cite{MZGZ09}, \cite{Z17}, and \cite{CCDO18} to extend the regularity proof to  our case. Note that, as in the case if the previous contributions mentioned above, we don't have a nilpotent structure, so when we differentiate the equation we need to account for all commutators by relying on the root structure of the Lie algebra. \par

Given the technical character of the regularity proofs, we  present first the proof for $\SU(3)$ in full detail, and later indicate the minor modifications needed in the general case.

\section{Statements of the Main Results for $\SU(3)$}
\par The special unitary group of $3 \times 3$ complex matrices is defined by
$$\SU (3) = \{ g \in \GL (3, \comp) \, : \; g \cdot g^* = I \, , \; \det g = 1 \} \, ,$$
and its Lie algebra by
$$\su (3) = \{ X \in \gl (3, \comp) \, : \; X + X^* = 0 \, , \; \tr X =0 \} \, .$$
The inner product is defined by  a multiple of the Killing form
$$\langle X , Y \rangle = - \frac{1}{2} \tr ( X  Y ) \, .$$
We consider the two-dimensional maximal torus
$$\To = \left\{  \begin{pmatrix}  e^{ia_1} & 0 & 0 \\ 0 & e^{ia_2} & 0 \\ 0 & 0 & e^{ia_3} \end{pmatrix} \; : \; a_1, \, a_2, \, a_3 \in \real \, , \; a_1 + a_2 + a_3 = 0 \right\}$$
and its Lie algebra 
$$\T = \left\{  \begin{pmatrix} ia_1 & 0 & 0 \\ 0& ia_2 & 0 \\ 0&0& ia_3 \end{pmatrix} \; : \; a_1, \, a_2, \, a_3 \in \real \,  , \; a_1 + a_2 + a_3 = 0 \right\},$$
 which is our choice for the Cartan subalgebra.
The following are the Gell-Mann matrices, which form an
orthonormal basis of $\su (3)$:

\begin{align*}
&T_1 = \begin{pmatrix} -i  & 0 & 0 \\ 0& i  & 0 \\ 0&0& 0 \end{pmatrix}, 
&&T_2 = \begin{pmatrix} \frac{-i}{\sqrt{3}}  & 0 & 0 \\ 0& \frac{-i}{\sqrt{3}}  & 0 \\ 0&0& \frac{2i}{\sqrt{3}} \end{pmatrix}, \\[0.2cm]
&X_1 = \begin{pmatrix} 0  & 1 & 0 \\ -1 & 0  & 0 \\ 0&0& 0 \end{pmatrix}, 
&&X_2 = \begin{pmatrix} 0 & i & 0 \\  i & 0 & 0 \\ 0&0& 0 \end{pmatrix},
\\[0.2cm]
&X_3 = \begin{pmatrix} 0  & 0 & 0 \\ 0 & 0  & 1 \\ 0&-1& 0 \end{pmatrix}, 
&&X_4 = \begin{pmatrix} 0 & 0 & 0 \\  0 &  0 & -i \\ 0&-i& 0 \end{pmatrix},
\\[0.2cm]
&X_5 = \begin{pmatrix} 0  & 0 & 1 \\ 0 & 0  & 0 \\ -1&0& 0 \end{pmatrix},  
&&X_6 = \begin{pmatrix} 0 & 0 & i \\  0 &  0 & 0 \\ i&0& 0 \end{pmatrix}.
\end{align*}

For the method of Riemannian approximation, described in Section 3, the following two vector fields provide simpler calculations than $T_1$ and $T_2$. As it is described in Section \S 5, these are two of the positive roots.

\begin{equation*}
X_7 = -[X_1, X_2] = \begin{pmatrix} - 2i  & 0 & 0 \\ 0 & 2i  & 0 \\ 0&0& 0 \end{pmatrix}, \;
X_8 = -[X_3, X_4] = \begin{pmatrix} 0 & 0 & 0 \\  0 &  2i & 0 \\  0 & 0 & -2i \end{pmatrix}  \, .
\end{equation*}

We list all the commutators of the vector fields $X_1, ... , X_8$ in the next table.

{\footnotesize
\begin{table}[h]
\caption{Commutators in $\SU (3)$} 
\centering
\begin{tabular}{c||c|c|c|c|c|c|c|c}
~ & $X_1$ & $X_2$ & $X_3$ & $X_4$ & $X_5$ & $X_6$ & $\textcolor{red}{X_7}$ & $\textcolor{red}{X_8}$ \\ [0.5ex]
\hline \hline 
$X_1$ & $0$ & $-\textcolor{red}{X_7}$ & $\; \; \; X_5$ & $- X_6$ & $-X_3$ & $\; \; \;  X_4$ &  $4 X_2$ & $2 X_2$ \\[0.5ex]
\hline
$X_2$ & $\; \; \; \textcolor{red}{X_7}$ & $0$ & $\; \; \; X_6$ & $\; \; \; X_5$ & $ - X_4$ & $- X_3$ &  $-4 X_1$ & $-2X_1$ \\[0.5ex]
\hline
$X_3$ & $- X_5$ & $- X_6$ & $0$ & $-\textcolor{red}{X_8}$ & $\; \; \; X_1$ & $\; \; \; X_2$ &  $2X_4$ & $4X_4$ \\[0.5ex]
\hline
$X_4$ & $\; \; \; X_6$ & $- X_5$ & $ \textcolor{red}{X_8}$ & $0$ & $\; \; \; X_2$ & $- X_1$ &  $-2X_3$ & $ -4X_3$ \\[0.5ex]
\hline
$X_5$ & $\; \; \; X_3$ & $\; \; \; X_4$ & $-X_1$ & $-X_2$ & $0$ & $\textcolor{red}{X_8-X_7}$ &  $2X_6$ & $-2X_6$ \\[0.5ex]
\hline
$X_6$ & $-X_4$ & $\; \; \; X_3$ & $ - X_2$ & $\; \; \; X_1$ & $\textcolor{red}{X_7-X_8}$ & $0$ &  $-2X_5$ & $2X_5$ \\[0.5ex]
\hline
$\textcolor{red}{X_7}$ & $-4X_2$ & $4 X_1$ & $-2X_4$ & $2X_3$ & $-2X_6$ & $2X_5$ &  $0$ & $0$ \\[0.5ex]
\hline
$\textcolor{red}{X_8}$ & $-2X_2$ & $2X_1$ & $-4X_4$ & $4X_3$ & $2X_6$ & $-2X_5$ &  $0$ & $ 0$ \\[0.5ex]
\hline

\end{tabular}
\label{tab:hresult}
\end{table} }


In case of $\SU (3)$ the orthonormal basis for the horizontal subspace $\Ho$ is
$${\mathcal B}_{\Ho} = \{ X_1, X_2 , X_3, X_4, X_5, X_6 \} \, .$$
The commutation properties in Table \ref{tab:hresult} show that, by identifying $\Lie$ with the Lie algebra of left-invariant vector fields, 
${\mathcal B}_{\Ho}$ satisfies the H\"ormander condition and 
generates the horizontal distribution of a sub-Riemannian manifold.

Recall that the curve $\gamma : [0,T] \to \G$ is  subunitary  associated to ${\mathcal B}_{\Ho}$ if $\gamma$ is an absolutely continuous function, such that for all $i \in \{1,...,6\}$ there exists $\alpha_i \in L^{\infty} [0,T]$ with the properties 
$$\gamma'(t) = \sum_{i=1}^6 \alpha_i (t) \, X_i (\gamma(t)) \, , \; 
\sum_{i=1}^6 \alpha_i^2 (t) \leq 1 \, , \; \text{a.e.} \; t \in [0,T] \, .$$ 
The control distance (Carnot-Carath\'{e}odory distance) with respect to ${\mathcal B}_{\Ho}$  is defined by  
\begin{multline}\label{def:distance}
d(x,y) = \inf \{ T \geq 0 \, : \; \text{there exists $\gamma :[0,T] \to \G$, 
a subunitary curve}\\ \text{for ${\mathcal B}_{\Ho}$, connecting $x$ and $y$} \} \, .
\end{multline} 
We use $B_r$ for the Carnot-Carath\'{e}odory balls of radius $r$ generated by 
$d$.\\

Let us fix a bi-invariant Haar-measure and note that for left-invariant vector fields we always have $X_i^* = -X_i$. 
Consider a domain $\Omega \subset \SU(3)$, and the following quasilinear subelliptic equation:
\begin{equation} \label{eq:PDE1}
\sum_{i=1}^6 X_i \left( a_i (\he u) \right) = 0 \, , \;
 \mbox{in} \; \Omega \; ,
\end{equation} 
where for some $0 \leq \delta \leq 1$, $p > 1$ , $0 < l < L$, and for all $\eta, \xi \in {\mathbb R}^{6}$ the following properties hold:
\begin{align}\label{elliptic1}
\sum_{i,j=1}^6 \frac{\partial a_i}{\partial \xi_j}(\xi ) \;
\eta_i \eta_j &
\geq l \Bigl( \delta + |\xi|^2 \Bigr)^{\frac{p-2}{2}} |\eta|^2 \, ,\\\label{elliptic2}
\sum_{i,j=1}^6 \left| \frac{\partial a_i}{\partial \xi_j} (\xi
)\right|
& \leq L \Bigl( \delta + |\xi|^2 \Bigr)^{\frac{p-2}{2}} \, ,\\\label{elliptic3}
|a_i (\xi )| & \leq  L \left( \delta + |\xi |^2
\right)^{\frac{p-1}{2}} \, .
\end{align}
The quintessential  representative example for the functions $a_i$ is given by 
$$a_i (\xi) = (\delta + |\xi|^2 )^{\frac{p-2}{2}} \xi_i \, .$$


A function $u \in W^{1,p}_{\Ho, \loc}(\Omega )$ is a weak solution of \eqref{eq:PDE1} if
\begin{equation}
\sum_{i=1}^6 \int_{\Omega } a_i ( \he u (x)) \; X_i \phi (x)
\,  dx = 0 \, , \; \text{for all $ \phi \in C_0^{\infty}(\Omega )$.}
\end{equation}

We list our main results:

\begin{theorem} \label{thm:Lipschitz}
Let $p > 1$ and $u \in W_{\Ho,\loc}^{1,p} (\Omega )$ be a weak solution of \eqref{eq:PDE1}.
Then there exists a constant $c>0$, depending only on $p, l, L$, such that for any Carnot-Carath\'eodory  ball $B_r \subset \subset \Omega$ we have 
\begin{equation}\label{ineq:Linfinity}
\sup_{B_{r/2}} |\he u | \leq c \left( -\hspace{-0.45cm}\int_{B_r} (\delta + |\he u |^2 )^{\frac{p}{2}} dx \right)^{\frac{1}{p}} \, .
 	\end{equation}

\end{theorem}

\begin{theorem} \label{thm:GradientHolder}
Let $p \geq 2$ and $u \in W_{\Ho,\loc}^{1,p} (\Omega )$ be a weak solution of \eqref{eq:PDE1}. Then $\he u \in C^{\alpha}_{\loc} (\Omega)$.
\end{theorem}

\section{The proof of Theorem \ref{thm:Lipschitz}}\label{sec:5SU(3)}

Consider an arbitrary, but fixed $0 < \eps < 1$. Define the following vector fields:
\begin{itemize}
	\item For $i \in \{1,...,6\}$ define $X_i^{\eps} = X_i.$
	\item For $i \in \{7 ,8\}$ define $X_i^{\eps} = \eps X_i$.
\end{itemize}
Regarding the behavior as $\eps \to 0$, we have three types of commutators:
\begin{equation}\label{epsiloncom}
\begin{split}
& [\xe_1 , \xe_2 ] = -\frac{1}{\eps} \xe_7	, \; \; [\xe_3, \xe_4 ] =- \frac{1}{\eps} \xe_8, \; \; [\xe_5, \xe_6 ] = \frac{1}{\eps} (\xe_8-\xe_7 )\\
& [\xe_7 , \xe_1 ] = -4\eps \xe_2 , ... , \; [\xe_7 , \xe_3 ] = -2\eps \xe_4 , ... , [\xe_8 , \xe_1 ] = -2\eps\xe_2, ...\\
& [\xe_1 , \xe_3 ] = \xe_5, \; \; [\xe_1, \xe_4 ] = - \xe_6, ... , \; [\xe_2, \xe_3 ] = \xe_6, ...
\end{split}
\end{equation}

We will use the following notations:
\begin{itemize}
\item $\vv = (X_7, X_8 )$, $\he = (X_1, ... , X_6 )$. 
\item $\ve = (\xe_7, \xe_8 )$, $\nabla^{\eps} = (\xe_1 , ... , \xe_6,\xe_7,\xe_8 )$.
\item $\we = \delta + |\nabla^{\eps} \ue |^2$. 	
\end{itemize}
We can always extend the vector function $(a_1, ... , a_6 )$ to 
$(a_1, ... , a_8 )$ in such a way that we keep the properties
\eqref{elliptic1}, \eqref{elliptic2} and \eqref{elliptic3}. Consider the quasilinear elliptic PDE, which will serve as a Riemannian approximation of \eqref{eq:PDE1}:
\begin{equation}\label{eq:PDE2}
\sum_{i=1}^8 X_i^{\eps} (a_i (\nae u )) = 0 , \; \; \text{in} \; \; \Omega .
\end{equation}

\begin{remark}
If $\delta >0$ and $\eps >0$, the weak solutions of the non-degenerate quasilinear elliptic equation \eqref{eq:PDE2} are smooth in $\Omega$ by classical regularity theory. See for example \cite{LU68}.
\end{remark}

The series of lemmas that follow contain generalizations of the Cacciopoli-type inequalities that  were developed and gradually refined in the case of Heisenberg group in \cite{MM07,MZGZ09,Z17, R15, CCDO18}.  
\begin{lemma}\label{lemma:vertical1}
Let $0< \delta <1$, $\beta \geq 0$ and $\eta \in C_0^{\infty} (\Omega )$ be such that $0 \leq \eta \leq 1$.
Then there exists a constant $c>0$ depending only on  $p$, $l$ and $L$ such that for any solution $\ue \in C^{\infty} (\Omega )$ of \eqref{eq:PDE2} we have
\begin{multline}\label{ineq:vertical1}
\int_{\Omega} \eta^2 \, \we^{\frac{p-2}{2}} \, |\ve \ue|^{2\beta} \,
|\nae \ve \ue |^2 \, dx \\
\leq c \int_{\Omega} |\nae \eta|^2 \,  \we^{\frac{p-2}{2}} \, |\ve \ue |^{2\beta+2} \, dx\\
 + c \eps^2 (\beta + 1)^2 \int_{\Omega}	\eta^2 \, \we^{\frac{p}{2}} \, |\ve \ue|^{2\beta} \, dx .
\end{multline}
\end{lemma}

\begin{proof}
In order to accommodate all the terms, we will simplify the writing of \eqref{eq:PDE2}:
\begin{equation}\label{eq:PDE3}
\sum_i X_i^{\eps} (a_i) = 0 \, .
\end{equation}
By differentiating \eqref{eq:PDE3} with respect to $\xe_7$ and
switching $\xe_7$ and $\xe_i$ we get
\begin{multline*}
\sum_i \xe_i (\xe_7 (a_i)) = 4\eps \xe_2 (a_1) -4 \eps \xe_1 (a_2) + 2\eps \xe_4 (a_3) -2\eps \xe_3 (a_4) \\+2\eps \xe_6 (a_5)-2\eps\xe_5 (a_6) .
\end{multline*}
Using the notation $a_{ij} = \frac{\partial a_i}{\partial \xi_j}$, for any $\phi \in C_0^{\infty} (\Omega )$ we get
$$\sum_{i,j} \int_{\Omega} a_{ij} \, \xe_7 \xe_j\ue \, \xe_i \phi \, dx = 4 \eps \int_{\Omega} a_1 \, \xe_2 \, \phi \, dx + \; \; \text{similar terms} .$$
Another switch between $\xe_7$ and $\xe_j$ leads to
\begin{multline}\label{eq:PDE4}
\sum_{i,j} \int_{\Omega} a_{ij}	\, \xe_j\xe_7\ue \, \xe_i\phi \, dx \\
= 4 \eps \int_{\Omega} a_1 \, \xe_2\phi \, dx + \; \text{similar terms}\\
+4 \eps \sum_i \int_{\Omega} a_{i1} \, \xe_2\ue \, \xe_i \phi \, dx + \; \text{similar terms}.
\end{multline}

Let us use $\phi = \eta^2 \, |\ve \ue|^{2\beta} \, \xe_7\ue$ in \eqref{eq:PDE4}. 
Then,
\begin{multline*}
\xe_i \phi = 2\eta \, \xe_i\eta \, |\ve \ue |^{2\beta} \, \xe_7 \ue\\ + \eta^2 \, \beta \, |\ve \ue |^{2\beta-2} \, \xe_i (|\ve\ue|^2) \, \xe_7\ue \\
+\eta^2 |\ve\ue|^{2\beta} \, \xe_i\xe_7\ue \, ,
\end{multline*}                                                                                                           
and hence
\begin{multline}\label{eq:diffx3}
\sum_{i,j} \int_{\Omega} a_{ij} \, \xe_j \xe_7\ue \, 2 \eta \xe_i \eta \, |\ve\ue|^{2\beta}\, \xe_7\ue\, dx\\
+ \sum_{i,j} \int_{\Omega} a_{ij} \, \xe_j \xe_7 \ue \, \eta^2 \, \beta |\ve\ue|^{2\beta-2} \, \xe_i (|\ve\ue|^2)\, \xe_7 \ue \, dx\\
+ \sum_{i,j} \int_{\Omega} a_{ij} \, \xe_j \xe_7\ue \, \eta^2 \, |\ve\ue|^{2\beta}\, \xe_i\xe_7\ue\, dx\\
\shoveleft{ = 4\eps \int_{\Omega}} a_1 \, 2\eta \, \xe_2\eta \, |\ve\ue|^{2\beta} \, \xe_7\ue \, dx\\
+4 \eps \int_{\Omega} a_1 \, \eta^2 \, \beta |\ve\ue|^{2\beta-2} \, \xe_2(|\ve\ue|^2) \, \xe_7 \ue \, dx\\
+4 \eps \int_{\Omega} a_1 \, \eta^2 \, |\ve\ue|^{2\beta} \, \xe_2\xe_7 \ue \, dx + ... \\
\shoveleft{ \; \; \; +4 \eps \sum_i \int_{\Omega} a_{i1} \, \xe_2\ue \, 2\eta \, \xe_i\eta \, |\ve\ue|^{2\beta} \, \xe_7\ue \, dx}\\
 +4\eps \sum_i \int_{\Omega} a_{i1} \, \xe_2\ue \, \eta^2 \, \beta \, |\ve\ue|^{2\beta-2} \, \xe_i (|\ve\ue|^2) \, \xe_7\ue \, dx\\
+4\eps \sum_i \int_{\Omega} a_{i1} \, \xe_2\ue \, \eta^2 \, |\ve\ue|^{2\beta} \, \xe_i\xe_7\ue \, dx + ... \\
 \end{multline}	
As we already did in \eqref{eq:diffx3}, in the following estimates we will list one member of each group of terms requiring certain type of inequalities and signal the presence of similar terms by \lq\lq...\rq\rq.
By writing an identical equation for $\xe_8$ and adding it to \eqref{eq:diffx3}, we get nine representative terms:
\begin{multline*}
(L_1) + (L_2) +(L_3)\\
= (R_{11}) + (R_{12}) + (R_{13}) + ...\\
+ (R_{21}) + (R_{22}) + (R_{23}) + ...
\end{multline*}

We estimate each term.
\begin{multline*}	
\shoveright{(L_3) \geq l \int_{\Omega} \eta^2 \, \we^{\frac{p-2}{2}} \, |\ve\ue|^{2\beta} \, |\nae\ve\ue|^2 \, dx.}
\end{multline*}
\begin{multline*}
(L_2) = \frac{1}{2} \sum_{i,j} \int_{\Omega} a_{ij} \, \xe_j(|\ve\ue|^2) \, \eta^2 \, \beta \, |\ve\ue|^{2\beta-2} \, \xe_i (|\ve\ue|^2) \, dx \\
\geq \frac{\beta l}{2} \int_{\Omega} \eta^2 \, \we^{\frac{p-2}{2}} \, |\ve\ue|^{2\beta-2} \, |\nae(|\ve\ue|^2)|^2 \, dx	
\end{multline*}
\begin{multline*}
(L_1) \leq c \int_{\Omega} \we^{\frac{p-2}{2}} \, |\nae\ve\ue| \, 2\eta\, |\nae\eta| \, |\ve\ue|^{2\beta+1} \, dx\\
\leq \frac{l}{100} \int_{\Omega} \eta^2 \we^{\frac{p-2}{2}} \, |\ve\ue|^{2\beta}|\nae\ve\ue|^2 \, dx\\
+ c \int_{\Omega} |\nae\eta|^2 \, \we^{\frac{p-2}{2}} \, |\ve\ue|^{2\beta+2} \, dx	
\end{multline*}
\begin{multline*}
(R_{11})+(R_{21}) \leq c \eps \int_{\Omega} \we^{\frac{p-1}{2}} \, \eta \, |\nae\eta| \, |\ve\ue|^{2\beta+1} \, dx \\
\leq c \int_{\Omega} |\nae\eta|^2 \, \we^{\frac{p-2}{2}}	 \, |\ve\ue|^{2\beta+2} \, dx\\
+ c \eps^2 \int_{\Omega} \eta^2 \, \we^{\frac{p}{2}} \, |\ve\ue|^{2\beta} \, dx
\end{multline*}
\begin{multline*}
(R_{12})+(R_{13})+(R_{22})+(R_{23})\\ \leq c\eps(\beta+1) \int_{\Omega} \we^{\frac{p-1}{2}} \, \eta^2 \, |\ve\ue|^{2\beta} \, |\nae\ve\ue| \, dx\\
	\leq \frac{l}{100} \int_{\Omega} \eta^2 \, \we^{\frac{p-2}{2}} \, |\ve\ue|^{2\beta} \, |\nae\ve\ue|^2 \, dx\\
	+ c\eps^2 (\beta+1)^2 \int_{\Omega} \eta^2 \, \we^{\frac{p}{2}} \, |\ve\ue|^{2\beta} \, dx
\end{multline*}
By combining all these estimates we get \eqref{ineq:vertical1}.
\end{proof}

\begin{remark}
If in Lemma (\ref{lemma:vertical1}) we change $\eta$ to $\eta^{\beta +2}$ we get the following estimate:

\begin{multline}\label{ineq:vertical1.1}
\int_{\Omega} \eta^{2\beta+4} \, \we^{\frac{p-2}{2}} \, |\ve \ue|^{2\beta} \,
|\nae \ve \ue |^2 \, dx \\
\leq c (\beta+1)^2 ||\nae\eta||^2_{L^{\infty}} \int_{\Omega} \eta^{2\beta+2} \,  \we^{\frac{p-2}{2}} \, |\ve \ue |^{2\beta+2} \, dx\\
 + c \eps^2 (\beta + 1)^2 \int_{\Omega}	\eta^{2\beta+4} \, \we^{\frac{p}{2}} \, |\ve \ue|^{2\beta} \, dx .
\end{multline}
\end{remark}

\begin{lemma}\label{lemma:vertical2}
Let $0< \delta <1$, $\beta \geq 0$ and $\eta \in C_0^{\infty} (\Omega )$ be such that $0\leq \eta \leq 1$.
Then there exists a constant $c>0$ depending only on $\G$, $p$, $l$ and $L$ such that for any solution $\ue \in C^{\infty} (\Omega )$ of \eqref{eq:PDE2} we have
\begin{multline}\label{ineq:vertical2}
\int_{\Omega} \eta^2 \, \we^{\frac{p-2}{2}+\beta} \,
|\nae \nae \ue |^2 \, dx \\
\leq c (\beta +1)^4 \int_{\Omega} \eta^2 \,  \we^{\frac{p-2}{2}+\beta} \, |\vv \ue |^{2} \, dx\\
 + c (\beta+1)^2 \int_{\Omega}	(\eta^2+|\nae \eta|^2 + \eta|\vv\eta| )\, \we^{\frac{p}{2}+\beta} \, dx .
\end{multline}
\end{lemma}
	
\begin{proof}
Let's differentiate equation \eqref{eq:PDE3} with respect to $\xe_1$ and switch $\xe_1$ and $\xe_i$. In this way we get
\begin{equation*}
\sum_i \xe_i (\xe_1 a_i ) = \frac{1}{\eps} \xe_7 a_2 -4 \eps \xe_2 a_7 - \xe_5 a_3 
+ \text{similar terms} .
\end{equation*}
The weak form of this equation looks like
\begin{multline}\label{eq:vertical2}
\sum_{i,j} \int_{\Omega} a_{ij} \, \xe_1 \xe_j \ue \, \xe_i \phi \, dx\\ \shoveleft{=
	 \frac{1}{\eps} \int_{\Omega} a_2 \, \xe_7 \phi \, dx -4 \eps \int_{\Omega} a_7 \, \xe_2 \phi \, dx  - \int_{\Omega} a_3 \, \xe_5 \phi \, dx + ... }
\end{multline}
After switching $\xe_j$ and $\xe_1$ in \eqref{eq:vertical2} we get
\begin{multline}\label{eq:PDE5}
\sum_{i,j} \int_{\Omega} a_{ij} \, \xe_j \xe_1 \ue \, \xe_i \phi \, dx\\ 
	\shoveleft{ = \frac{1}{\eps} \int_{\Omega} a_2 \, \xe_7 \phi \, dx -4\eps \int_{\Omega} a_7 \, \xe_2 \phi \, dx 
	-  \int_{\Omega} a_3 \, \xe_5 \phi \, dx + ...}\\
	\shoveleft{+ \frac{1}{\eps} \sum_i \int_{\Omega} a_{i2} \, \xe_7 \ue \, \xe_i \phi \, dx 
	 -4 \eps \sum_i \int_{\Omega} a_{i7} \, \xe_2 \ue \, \xe_i \phi \, dx} \\
	\shoveright{- \sum_i \int_{\Omega} a_{i3} \, \xe_5 \ue \, \xe_i\phi \, dx + ... }\\
\end{multline}

Let us use $\phi = \eta^2 \, \we^{\beta} \, \xe_1\ue$ in \eqref{eq:PDE5}.  
\begin{multline*}
	\sum_{i,j} \int_{\Omega} a_{ij}\, \xe_j\xe_1\ue\, \eta^2 \, \we^{\beta} \, \xe_i \xe_1 \ue \, dx\\
	+ \sum_{i,j} \int_{\Omega} a_{ij} \, \xe_j\xe_1\ue \, \eta^2 \, \beta \we^{\beta-1} \, \xe_i (|\nae\ue|^2) \, \xe_1\ue\, dx\\
	\shoveright{+ \sum_{i,j} \int_{\Omega} a_{ij} \, \xe_j\xe_1\ue \, 2\eta\, \xe_i\eta \, \we^{\beta} \, \xe_1\ue \, dx + ...}\\
\shoveleft{= \frac{1}{\eps} \int_{\Omega} a_2 \, \eta^2 \, \we^{\beta} \, \xe_7\xe_1\ue \, dx}\\
+\frac{1}{\eps} \int_{\Omega} a_2 \, \eta^2 \, \beta \we^{\beta-1} \, \xe_7 (|\nae\ue|^2) \, \xe_1 \ue \, dx\\
\shoveright{+\frac{1}{\eps} \int_{\Omega} a_2 \, 2\eta \, \xe_7\eta \, \we^{\beta} \, \xe_1\ue \, dx +...}\\
\shoveleft{-4 \eps \int_{\Omega} a_7 \, \eta^2 \, \we^{\beta} \, \xe_2\xe_1\ue \, dx}\\
-4 \eps \int_{\Omega} a_7 \, \eta^2 \, \beta\we^{\beta-1} \, \xe_2 (|\nae\ue|^2) \, \xe_1\ue \, dx\\
\shoveright{-4 \eps \int_{\Omega} a_7 \, 2\eta\, \xe_2\eta \, \we^{\beta} \, \xe_1\ue \, dx+...}\\
\shoveleft{-  \int_{\Omega} a_3 \, \eta^2 \, \we^{\beta} \, \xe_5\xe_1\ue \, dx}\\
- \int_{\Omega} a_3 \, \eta^2 \, \beta\we^{\beta-1} \, \xe_5 (|\nae\ue|^2) \, \xe_1\ue \, dx\\
\shoveright{- \int_{\Omega} a_3 \, 2\eta \, \xe_5\eta \, \we^{\beta} \, \xe_1\ue \, dx + ...}\\
\shoveleft{+\frac{1}{\eps} \sum_i \int_{\Omega} a_{i2} \, \xe_7\ue \, \eta^2 \, \we^{\beta} \, \xe_i\xe_1\ue \, dx}\\
+\frac{1}{\eps} \sum_i \int_{\Omega} a_{i2} \, \xe_7\ue \, \eta^2 \, \beta \we^{\beta-1} \, \xe_i (|\nae\ue|^2) \, \xe_1 \ue \, dx\\
\shoveright{+\frac{1}{\eps} \sum_i \int_{\Omega} a_{i2} \, \xe_7\ue \, 2\eta \, \xe_i\eta \, \we^{\beta} \, \xe_1\ue \, dx +...}\\
\shoveleft{-4\eps \sum_i \int_{\Omega} a_{i7} \, \xe_2\ue \, \eta^2 \, \we^{\beta} \, \xe_i\xe_1\ue \, dx}\\
-4\eps \sum_i \int_{\Omega} a_{i7} \, \xe_2\ue \, \eta^2 \, \beta \we^{\beta-1} \, \xe_i (|\nae\ue|^2) \, \xe_1 \ue \, dx\\
\shoveright{-4\eps \sum_i \int_{\Omega} a_{i7} \, \xe_2\ue \, 2\eta \, \xe_i\eta \, \we^{\beta} \, \xe_1\ue \, dx +...}\\
\shoveleft{- \sum_i \int_{\Omega} a_{i3} \, \xe_5\ue \, \eta^2 \, \we^{\beta} \, \xe_i\xe_1\ue \, dx}\\
- \sum_i \int_{\Omega} a_{i3} \, \xe_5\ue \, \eta^2 \, \beta \we^{\beta-1} \, \xe_i (|\nae\ue|^2) \, \xe_1 \ue \, dx\\
\shoveright{- \sum_i \int_{\Omega} a_{i3} \, \xe_5\ue \, 2\eta \, \xe_i\eta \, \we^{\beta} \, \xe_1\ue \, dx +... \, .}\\
\end{multline*}
Repeat the above calculations for $\xe_2 , ... , \xe_8$ and add all equations. In this way we get an equation in the following format
\begin{equation*}
	\text{L(1.1)} + \text{L(1.2)} + \text{L(1.3)}\\
	= \sum_{i=1}^6 \text{R(i.1)} + \text{R(i.2)} + \text{R(i.3)} + ... 
\end{equation*}
We estimate each term.
\begin{multline*}
\text{L(1.1)} = \sum_{i,j,k} \int_{\Omega} a_{ij}\, \xe_j\xe_k\ue\, \eta^2 \, \we^{\beta} \, \xe_i \xe_k \ue \, dx \\ \geq l \int_{\Omega} \eta^2 \, \we^{\frac{p-2}{2}+\beta} \, |\nae\nae\ue|^2 \, dx.
\end{multline*}
\begin{multline*}
\text{L(1.2)} = \sum_{i,j} \int_{\Omega} a_{ij} \, \sum_k \xe_j\xe_k\ue \, \xe_k\ue \, \eta^2 \, \beta \we^{\beta-1} \, \xe_i (|\nae\ue|^2) \,  dx\\
= \frac{\beta}{2} \sum_{i,j} \int_{\Omega} a_{ij} \, \xe_j (|\nae \ue|^2) \, \eta^2 \, \we^{\beta-1} \, \xe_i (|\nae\ue|^2) \,  dx\\
\geq \frac{\beta l}{2} \int_{\Omega} \eta^2 \we^{\frac{p-2}{2}+\beta-1}\, |\nae (|\nae\ue|^2) \, dx.
\end{multline*}
\begin{multline*}
	|\text{L(1.3)}| \leq c \int_{\Omega} \we^{\frac{p-2}{2}} \, |\nae\nae\ue| \, \eta \, |\nae\eta| \, \we^{\beta + \frac{1}{2}} \, dx\\
	\leq \frac{l}{100} \int_{\Omega} \eta^2 \, \we^{\frac{p-2}{2}+\beta} \, |\nae\nae\ue|^2 \, dx \\ + c \int_{\Omega} |\nae\eta|^2 \, \we^{\frac{p}{2} + \beta} \, dx \, .
	\end{multline*}
\begin{multline*}
	\text{R(1.1)} = \frac{1}{\eps} \int_{\Omega} a_2 \, \eta^2 \, \we^{\beta} \, \xe_7\xe_1\ue \, dx + ...\\ 
	= \frac{1}{\eps} \int_{\Omega} a_2 \, \eta^2 \, \we^{\beta} \, (\xe_1\xe_7\ue -4 \eps \xe_2\ue )\, dx + ...\\
	= -\frac{1}{\eps} \int_{\Omega} \xe_1 ( a_2 \, \eta^2 \, \we^{\beta}) \, \xe_7\ue \, dx +4 \int_{\Omega} a_2 \, \eta^2 \, \we^{\beta} \, \xe_2 \ue \, dx+ ...\\
	\shoveleft{=- \sum_i \int_{\Omega} a_{2i} \, \xe_1\xe_i\ue \, \eta^2 \, \we^{\beta} \, X_7\ue \, dx- \int_{\Omega} a_2 \, 2\eta\, \xe_1\eta \, \we^{\beta} \, X_7\ue \, dx}\\	
	-\int_{\Omega} a_2 \, \eta^2 \, \beta \we^{\beta-1} \, 2 \langle \nae\ue , \xe_1 \nae\ue \rangle \, X_7\ue \, dx\\
	+4\int_{\Omega} a_2 \, \eta^2 \, \we^{\beta} \, \xe_2\ue \, dx + ...\\
	\shoveleft{\leq c \int_{\Omega} \we^{\frac{p-2}{2}} \, |\nae\nae\ue| \, \eta^2 \we^{\beta} \, |\vv \ue| \, dx  
	+ c \int_{\Omega} \we^{\frac{p-1}{2}} \, \eta \, |\nae\eta| \, \we^{\beta} \, |\vv\ue| \, dx}\\
	+ c \int_{\Omega} \we^{\frac{p-1}{2}} \, \eta^2 \, \beta\we^{\beta-1} \, \we^{\frac{1}{2}} \, |\nae\nae\ue| \, |\vv\ue| \, dx\\
	+ c \int_{\Omega} \we^{\frac{p-1}{2}} \, \eta^2 \, \we^{\beta} \, \we^{\frac{1}{2}} \, dx\\
	\shoveleft{ \leq \frac{l}{200} \int_{\Omega} \eta^2 \, \we^{\frac{p-2}{2}+\beta} \, |\nae\nae\ue|^2 \, dx + 2c \int_{\Omega} \eta^2 \, \we^{\frac{p-2}{2}+\beta} \, |\vv \ue|^2 \, dx}\\ 
	+ c \int_{\Omega} |\nae\eta|^2 \, \we^{\frac{p}{2}+\beta} \, dx\\
	\shoveleft{ + \frac{l}{200} \int_{\Omega} \eta^2 \, \we^{\frac{p-2}{2}+\beta} \, |\nae\nae\ue|^2 \, dx + c \beta^2 \int_{\Omega} \eta^2 \, \we^{\frac{p-2}{2}+\beta} \, |\vv\ue|^2 \, dx}\\ 
	+ c \int_{\Omega} \eta^2 \, \we^{\frac{p}{2}+\beta} \, dx\\
	\shoveleft{ \leq \frac{l}{100} \int_{\Omega} \eta^2 \, \we^{\frac{p-2}{2}+\beta} \, |\nae\nae\ue|^2 \, dx + c (\beta+1)^2 \int_{\Omega} \eta^2 \, \we^{\frac{p-2}{2}+\beta} \, |\vv \ue|^2 \, dx}\\ 
	+ c \int_{\Omega} (\eta^2 + |\nae\eta|^2 )\, \we^{\frac{p}{2}+\beta} \, dx \, .\\
\end{multline*}

For  the next set of estimates we will use the following identity that comes from the commutators' Table \ref{tab:hresult}:
$$\langle \nae\ue , \xe_i\nae\ue \rangle = \langle \nae\ue , \nae \xe_i\ue \rangle \, , \; \text{if} \; i = 7 \; \text{or} \; 8 \, .$$

\begin{multline*}
\text{R(1.2)} = \frac{1}{\eps} \int_{\Omega} a_2 \, \eta^2 \, \beta \we^{\beta-1} \, \xe_7 (|\nae\ue|^2) \, \xe_1 \ue \, dx+...\\
	= \frac{1}{\eps} \int_{\Omega} a_2 \, \eta^2 \, \beta \we^{\beta-1} \, 2\langle \nae\ue , \xe_7\nae\ue \rangle \, \xe_1 \ue \, dx+...\\
\shoveright{= \frac{1}{\eps} \int_{\Omega} a_2 \, \eta^2 \, \beta \we^{\beta-1} \, 2\langle \nae\ue , \nae\xe_7\ue \rangle \, \xe_1 \ue \, dx+...}\\
\shoveleft{ =  \frac{2\beta}{\eps} \sum_i \int_{\Omega} a_2 \, \eta^2 \we^{\beta-1} \, \xe_i\ue \, \xe_i\xe_7\ue \, \xe_1\ue \, dx+...}\\
= -\frac{2\beta}{\eps} \sum_i \int_{\Omega} \xe_i ( a_2 \, \eta^2 \we^{\beta-1} \, \xe_i\ue \, \xe_1\ue )\, \xe_7\ue \, dx+...\\	
= -2\beta \sum_i \int_{\Omega} \xe_i ( a_2 \, \eta^2 \we^{\beta-1} \, \xe_i\ue \, \xe_1\ue )\, X_7\ue \, dx+...\\
\shoveleft{=-2\beta \sum_{i,j} \int_{\Omega} a_{2j} \, \xe_i\xe_j\ue \, \eta^2 \we^{\beta-1} \, \xe_i\ue \, \xe_1\ue \, X_7\ue \, dx+...}\\
-2\beta \sum_i \int_{\Omega} a_2 \, 2\eta \, \xe_i\eta\, \we^{\beta-1} \, \xe_i\ue \, \xe_1\ue \, X_7\ue \, dx+...\\
-2\beta(\beta-1) \sum_i \int_{\Omega} a_2 \, \eta^2 \, \we^{\beta-2} \, \xe_i (|\nae\ue|^2) \, \xe_i\ue \, \xe_1\ue \, X_7\ue \, dx+...\\
-2\beta \sum_i \int_{\Omega} a_2 \, \eta^2 \we^{\beta-1} \, \xe_i\xe_i\ue \, \xe_1\ue \, X_7\ue \, dx+...\\
-2\beta \sum_i \int_{\Omega} a_2 \, \eta^2 \we^{\beta-1} \, \xe_i\ue \, \xe_i\xe_1\ue \, X_7\ue \, dx+...\\
\shoveleft{\leq c \beta \int_{\Omega} \we^{\frac{p-2}{2}+\beta} \, |\nae\nae\ue| \, \eta^2 \, |\vv\ue| \, dx}\\
+ c \beta \int_{\Omega} \we^{\frac{p-1}{2}+\beta} \, \eta \, |\nae\eta| \, |\vv\ue| \, dx\\
+ c (\beta+1)^2 \int_{\Omega} \we^{\frac{p-2}{2}+\beta} \, |\nae\nae\ue| \, \eta^2 \, |\vv\ue| \, dx\\
\shoveleft{\leq \frac{l}{200} \int_{\Omega} \eta^2 \, \we^{\frac{p-2}{2}+\beta}\, |\nae\nae\ue|^2 \, dx + c\beta^2 \int_{\Omega} \eta^2 \, \we^{\frac{p-2}{2}+\beta} \, |\vv\ue|^2 \, dx}\\
+c \int_{\Omega} |\nae\eta|^2 \, \we^{\frac{p}{2}+\beta} \, dx + c\beta^2 \int_{\Omega} \eta^2 \, \we^{\frac{p-2}{2}+\beta} \, |\vv\ue|^2 \, dx\\
+\frac{l}{200} \int_{\Omega} \eta^2 \, \we^{\frac{p-2}{2}+\beta}\, |\nae\nae\ue|^2 \, dx + c(\beta+1)^4 \int_{\Omega} \eta^2 \, \we^{\frac{p-2}{2}+\beta} \, |\vv\ue|^2 \, dx\\
\shoveleft{\leq \frac{l}{100} \int_{\Omega} \eta^2 \, \we^{\frac{p-2}{2}+\beta}\, |\nae\nae\ue|^2 \, dx} \\ 
+ c(\beta+1)^4 \int_{\Omega} \eta^2 \, \we^{\frac{p-2}{2}+\beta} \, |\vv\ue|^2 \, dx + c \int_{\Omega} |\nae\eta|^2 \, \we^{\frac{p}{2}+\beta} \, dx \, .
\end{multline*}

\begin{multline*}
\shoveright{\text{R(1.3)+R(2.3)+R(5.3)} \leq c (\eps+1)  \int_{\Omega} \eta |\vv\eta| \, \we^{\frac{p}{2}+\beta} \, dx .}
\end{multline*}

\begin{multline*}
	\text{R(2.1)+R(2.2)+R(3.1)+R(3.2)+R(5.1)+R(5.2)+R(6.1)+R(6.2)} \\ \leq c (\eps+1) (\beta+1) \int_{\Omega} \eta^2 \, \we^{\frac{p-1}{2}+\beta} \, |\nae\nae\ue| \, dx \\
	\leq \frac{l}{100} \int_{\Omega} \eta^2 \, \we^{\frac{p-2}{2}+\beta} \, |\nae\nae\ue|^2 \, dx + c (\eps+1)^2 (\beta+1)^2 \int_{\Omega} \eta^2 \, \we^{\frac{p}{2}+\beta} \, dx \, .
\end{multline*}

\begin{multline*}
	\text{R(3.3)+R(6.3)} \leq c \int_{\Omega} \we^{\frac{p-1}{2}} \, \eta \, |\nae\eta| \, \we^{\beta+\frac{1}{2}} \, dx \\
	\leq c \int_{\Omega} (\eta^2 + |\nae\eta|^2) \, \we^{\frac{p}{2}+\beta} \, dx .
\end{multline*}

\begin{multline*}
	\text{R(4.1)+R(4.2)}\leq c (\beta+1) \int_{\Omega} \we^{\frac{p-2}{2}+\beta} \, |\vv\ue| \, \eta^2 \, |\nae\nae\ue| \, dx \\
	\leq \frac{l}{100} \int_{\Omega} \eta^2 \, \we^{\frac{p-2}{2}+\beta} \, |\nae\nae\ue|^2 \, dx + c (\beta+1)^2 \int_{\Omega} \eta^2 \, \we^{\frac{p-2}{2}+\beta} \, |\vv\ue|^2 \, dx.
\end{multline*}

\begin{multline*}
	\text{R(4.3)}\leq c \int_{\Omega} \we^{\frac{p-1}{2}+\beta} \, |\vv\ue| \, \eta \, |\nae\eta| \, \, dx \\
	\leq c \int_{\Omega} \eta^2 \, \we^{\frac{p-2}{2}+\beta} \, |\vv\ue|^2 \, dx
	+ c \int_{\Omega} |\nae\eta|^2 \, \we^{\frac{p}{2}+\beta} \,  dx \, .
	\end{multline*}
By adding the estimates from above we get \eqref{ineq:vertical2} and this finished the proof of Lemma \ref{lemma:vertical2}.
\end{proof}

\begin{lemma}\label{lemma:vertical3}
Let $0< \delta <1$, $\beta \geq 1$ and $\eta \in C_0^{\infty} (\Omega )$ be such that $0 \leq \eta \leq 1$.
Then there exists a constant $c>0$ depending only on $p$, $l$ and $L$ such that for any solution $\ue \in C^{\infty} (\Omega )$ of \eqref{eq:PDE2} we have
\begin{multline}\label{ineq:vertical3}
\int_{\Omega} \eta^{2\beta+2} \, \we^{\frac{p-2}{2}} \, |\ve\ue|^{2\beta} \,
|\nae \nae \ue |^2 \, dx \\
\leq c \eps^2 (\beta +1)^4 \| \nae\eta \|^2_{L^{\infty}} \int_{\Omega} \eta^{2\beta} \,  \we^{\frac{p}{2}} \, |\ve \ue |^{2\beta-2} \, |\nae\nae\ue|^2 dx \, .\\
 \end{multline}
\end{lemma}

\begin{proof}
Let us use $\phi = \eta^{2\beta+2} \, |\ve\ue|^{2\beta} \, \xe_1\ue$ in 	\eqref{eq:vertical2}.
First, let us organize the terms of $\xe_i \phi$ in the following way:
\begin{multline*}
	\xe_i\phi = \eta^{2\beta+2} \, |\ve\ue|^{2\beta} \, \xe_1\xe_i \ue 
	+\delta_{i2} \, \frac{1}{\eps} \, \eta^{2\beta+2} \, |\ve\ue|^{2\beta} \, \xe_7\ue + ...\\
	\shoveright{ -4 \delta_{i7} \, \eps \eta^{2\beta+2}\, |\ve\ue|^{2\beta} \, \xe_2\ue + ...
	- \delta_{i3} \, \eta^{2\beta+2} \, |\ve\ue|^{2\beta} \, \xe_5\ue + ...}\\
	\; \; \; \; \; + \eta^{2\beta+2} \, \beta \, |\ve\ue|^{2\beta-2} \, \xe_i (|\ve\ue|^2) \, \xe_1 \ue\\
	 + (2\beta+2) \eta^{2\beta+1} \, \xe_i \eta \, |\ve\ue|^{2\beta} \, \xe_1 \ue \, .
\end{multline*} 

Therefore, equation \eqref{eq:vertical2} has the following form.
\begin{multline*}
	\sum_{i,j} \int_{\Omega} a_{ij} \, \xe_1\xe_j\ue \, \eta^{2\beta+2} \, |\ve\ue|^{2\beta} \, \xe_1\xe_i\ue \, dx\\
	+\frac{1}{\eps} \sum_j \int_{\Omega} a_{2j} \, \xe_1\xe_j\ue \, \eta^{2\beta+2} \, |\ve\ue|^{2\beta} \, \xe_7\ue \, dx + ...\\
	-4 \eps \sum_j \int_{\Omega} a_{7j} \, \xe_1\xe_j\ue \, \eta^{2\beta+2} \, |\ve\ue|^{2\beta} \, \xe_2\ue \, dx + ...\\
	-  \sum_j \int_{\Omega} a_{3j} \, \xe_1\xe_j\ue \, \eta^{2\beta+2} \, |\ve\ue|^{2\beta} \, \xe_5\ue \, dx + ...\\
	\shoveleft{+ \beta \sum_{i,j} \int_{\Omega} a_{ij} \, \xe_1\xe_j\ue \, \eta^{2\beta+2} \, |\ve\ue|^{2\beta-2} \, \xe_i (|\ve\ue|^2) \, \, \xe_1\ue \, dx}\\
	+ 2(\beta+1) \sum_{i,j} \int_{\Omega} a_{ij} \, \xe_1\xe_j\ue \, \eta^{2\beta+1} \, \xe_i\eta \, |\ve\ue|^{2\beta} \, \xe_1\ue \, dx \\\shoveleft{= -\frac{1}{\eps} \int_{\Omega} \xe_7 a_2 \, \eta^{2\beta+2} \, |\ve\ue|^{2\beta} \, \xe_1\ue \, dx + ...}\\
	+4\eps \int_{\Omega} \xe_2 a_7 \, \eta^{2\beta+2} \, |\ve\ue|^{2\beta} \, \xe_1\ue \, dx+...\\
	+  \int_{\Omega} \xe_5 a_3 \, \eta^{2\beta+2} \, |\ve\ue|^{2\beta} \, \xe_1\ue \, dx +...	\, .
	\end{multline*}
By repeating this for $i = 1,2, ... , 8$ and adding the equations we get the following terms: 
$$\sum_{i=1}^6 (L_i) = \sum_{i=1}^3 (R_i) \, .$$
Once more, let's estimate each term.
\begin{multline*}
	(L_1) = \sum_k \sum_{i,j} \int_{\Omega} a_{ij} \, \xe_k\xe_j\ue \, \eta^{2\beta+2} \, |\ve\ue|^{2\beta} \xe_k\xe_i\ue \, dx\\
	\geq l \int_{\Omega} \eta^{2\beta+2} \, \we^{\frac{p-2}{2}} \, |\ve\ue|^{2\beta} \, |\nae\nae\ue|^2 \, dx.
\end{multline*}

\begin{multline*}
(L_2) = \frac{1}{\eps} \sum_j \int_{\Omega} a_{2j} \, \xe_1\xe_j\ue \, \eta^{2\beta+2} \, |\ve\ue|^{2\beta} \, \xe_7\ue \, dx + ...\\
=\frac{1}{\eps} \int_{\Omega} \xe_1 a_{2} \, \, \eta^{2\beta+2} \, |\ve\ue|^{2\beta} \, \xe_7\ue \, dx + ...\\
= -\frac{1}{\eps} \int_{\Omega} a_{2} \, \xe_1 ( \eta^{2\beta+2} \, |\ve\ue|^{2\beta} \, \xe_7\ue )\, dx + ...\\
\shoveleft{=- \frac{1}{\eps} \int_{\Omega} a_{2} \, (2\beta+2) \, \eta^{2\beta+1} \, \xe_1\eta \, |\ve\ue|^{2\beta} \, \xe_7\ue \, dx + ...}\\
- \frac{1}{\eps} \int_{\Omega} a_{2} \, \eta^{2\beta+2} \,\beta |\ve\ue|^{2\beta-2} \xe_1 (|\ve\ue|^2) \, \xe_7\ue \, dx + ...\\
-\frac{1}{\eps} \int_{\Omega} a_{2} \, \eta^{2\beta+2} \, |\ve\ue|^{2\beta} \, \xe_1\xe_7\ue \, dx + ...\\
\shoveleft{\leq \frac{c (\beta+1)}{\eps} \int_{\Omega} \eta^{2\beta+1} \, \we^{\frac{p-1}{2}} \, |\nae\eta| \, |\ve\ue|^{2\beta+1} \, dx}\\
+ \frac{c (\beta+1)}{\eps} \int_{\Omega} \eta^{2\beta+2} \, \we^{\frac{p-1}{2}} \, |\ve\ue|^{2\beta} \, |\nae\ve\ue| \, dx\\
\shoveleft{\leq \frac{l}{400\eps^2} \int_{\Omega} \eta^{2\beta+2} \, \we^{\frac{p-2}{2}} \, |\ve\ue|^{2\beta+2} \, dx}\\
+ c (\beta+1)^2 \int_{\Omega} \eta^{2\beta} \, \we^{\frac{p}{2}} \, |\nae\eta|^2 \, |\ve\ue|^{2\beta} \, dx\\
\shoveleft{+ \frac{l}{200c\eps^2 \, (\beta+1)^2 ||\nae\eta||^2_{L^{\infty}}} \int_{\Omega} \eta^{2\beta+4} \, \we^{\frac{p-2}{2}} \, |\ve\ue|^{2\beta} \, |\nae\ve\ue|^2 \, dx}\\
+c^3 (\beta+1)^4 ||\nae\eta||^2_{L^{\infty}} \int_{\Omega} \eta^{2\beta} \, \we^{\frac{p}{2}} \, |\ve\ue|^{2\beta} \, dx \, .
\end{multline*}

In the following we use \eqref{ineq:vertical1.1}, the inequalities 
$$|\ve\ue |^2 \leq 2 \eps^2 |\nae\nae\ue|^2, \; \; ||\eta||_{L^{\infty}} \leq 1 \, ,$$ and that without loss of generality we  can assume   $||\nae\eta||_{L^{\infty}} \geq 1$.

\begin{multline*}
(L_2) \leq \frac{l}{200} \int_{\Omega} \eta^{2\beta+2} \, \we^{\frac{p-2}{2}} \, |\ve\ue|^{2\beta} \, |\nae\nae\ue|^2 \, dx\\
+ c (\beta+1)^2 \, ||\nae\eta||^2_{L^{\infty}} \, \eps^2 \int_{\Omega} \eta^{2\beta} \, \we^{\frac{p}{2}} \, |\ve\ue|^{2\beta-2} \, |\nae\nae\ue|^2 \, dx\\
\shoveleft{+ \frac{l}{200\eps^2 } \int_{\Omega} \eta^{2\beta+2} \, \we^{\frac{p-2}{2}} \, |\ve\ue|^{2\beta+2} \, dx}\\
+ c \int_{\Omega} \eta^{2\beta+4} \, \we^{\frac{p}{2}} \, |\ve\ue|^{2\beta} \, dx\\
+c^3 (\beta+1)^4 ||\nae\eta||^2_{L^{\infty}} \, \eps^2 \int_{\Omega} \eta^{2\beta} \, \we^{\frac{p}{2}} \, |\ve\ue|^{2\beta-2} \, |\nae\nae\ue|^2 \, dx \\
\shoveleft{\leq \frac{l}{100} \int_{\Omega} \eta^{2\beta+2} \, \we^{\frac{p-2}{2}} \, |\ve\ue|^{2\beta} \, |\nae\nae\ue|^2 \, dx}\\
+c (\beta+1)^4 ||\nae\eta||^2_{L^{\infty}} \, \eps^2 \int_{\Omega} \eta^{2\beta} \, \we^{\frac{p}{2}} \, |\ve\ue|^{2\beta-2} \, |\nae\nae\ue|^2 \, dx .
\end{multline*}

\begin{multline*}
	(L_3)+(L_4)+(R_2)+(R_3) \\ 
	\shoveleft{\leq c (\eps+1) \int_{\Omega} \eta^{2\beta+2} \, \we^{\frac{p-1}{2}} \, |\ve\ue|^{2\beta} \, |\nae\nae\ue| \, dx}\\
	\shoveleft{\leq \frac{l}{100} \int_{\Omega} \eta^{2\beta+2} \, \we^{\frac{p-2}{2}} \, |\ve\ue|^{2\beta} \, |\nae\nae\ue|^2 \, dx}\\
	+ c (\eps+1)^2 \int_{\Omega} \eta^{2\beta+2} \, \we^{\frac{p}{2}} \, |\ve\ue|^{2\beta} \, dx\\
	\shoveleft{ \leq \frac{l}{100} \int_{\Omega} \eta^{2\beta+2} \, \we^{\frac{p-2}{2}} \, |\ve\ue|^{2\beta} \, |\nae\nae\ue|^2 \, dx}\\
	+ c \eps^2(\eps+1)^2 \int_{\Omega} \eta^{2\beta+2} \, \we^{\frac{p}{2}} \, |\ve\ue|^{2\beta-2} \, |\nae\nae\ue|^2 \, dx.
	\end{multline*}

\begin{multline*}
	(L_5) \leq c \beta \int_{\Omega} \eta^{2\beta+2} \, \we^{\frac{p-1}{2}} \, |\ve\ue|^{2\beta-1} \, |\nae\nae\ue| \, |\nae\ve\ue| \, dx\\
	\shoveleft{ \leq \frac{l}{200 c\eps^2 (\beta+1)^2 ||\nae\eta||^2_{L^{\infty}}} \int_{\Omega} \eta^{2\beta+4} \, \we^{\frac{p-2}{2}} \, |\ve\ue|^{2\beta} \, |\nae\ve\ue|^2 \, dx}\\
	 + c^3 \eps^2 (\beta+1)^4 ||\nae\eta||^2_{L^{\infty}} \int_{\Omega} \eta^{2\beta} \,\we^{\frac{p}{2}} \, |\ve\ue|^{2\beta-2} \, |\nae\nae\ue|^2 \, dx\\
\shoveleft{ \leq \frac{l}{200 \eps^2} \int_{\Omega} \eta^{2\beta+2} \, \we^{\frac{p-2}{2}} \, |\ve\ue|^{2\beta+2} \, dx}\\
+ c \int_{\Omega} \eta^{2\beta+4} \, \we^{\frac{p}{2}} \, |\ve\ue|^{2\beta} \, dx\\
	 + c^3 \eps^2 (\beta+1)^4 ||\nae\eta||^2_{L^{\infty}} \int_{\Omega} \eta^{2\beta} \,\we^{\frac{p}{2}} \, |\ve\ue|^{2\beta-2} \, |\nae\nae\ue|^2 \, dx\\	 
\shoveleft{ \leq \frac{l}{100} \int_{\Omega} \eta^{2\beta+2} \, \we^{\frac{p-2}{2}} \, |\ve\ue|^{2\beta} \, |\nae\nae\ue|^2 \, dx}\\
	 + c \eps^2 (\beta+1)^4 ||\nae\eta||^2_{L^{\infty}} \int_{\Omega} \eta^{2\beta} \,\we^{\frac{p}{2}} \, |\ve\ue|^{2\beta-2} \, |\nae\nae\ue|^2 \, dx.
	 \end{multline*}
	
\begin{multline*}
	(L_6) \leq c (\beta+1) \int_{\Omega} \eta^{2\beta+1} \, |\nae\eta| \, \we^{\frac{p-1}{2}} \, |\ve\ue|^{2\beta} \, |\nae\nae\ue| \, dx\\
\shoveleft{ \leq \frac{l}{100} \int_{\Omega} \eta^{2\beta+2} \, \we^{\frac{p-2}{2}} \, |\ve\ue|^{2\beta} \, |\nae\nae\ue|^2 \, dx}\\
	 + c \eps^2 (\beta+1)^2 ||\nae\eta||^2_{L^{\infty}} \int_{\Omega} \eta^{2\beta} \,\we^{\frac{p}{2}} \, |\ve\ue|^{2\beta-2} \, |\nae\nae\ue|^2 \, dx. 
	\end{multline*}	
	
\begin{multline*}
	(R_1) = -\frac{1}{\eps} \sum_j \int_{\Omega} a_{2j} \, \xe_7\xe_j\ue \, \eta^{2\beta+2} \, |\ve\ue|^{2\beta} \, \xe_1\ue \, dx + ...\\
	\shoveleft{=-\frac{1}{\eps} \sum_j \int_{\Omega} a_{2j} \, \xe_j\xe_7\ue \, \eta^{2\beta+2} \, |\ve\ue|^{2\beta} \, \xe_1\ue \, dx + ...}\\
	+4 \int_{\Omega} a_{21} \, \xe_2\ue \, \eta^{2\beta+2} \, |\ve\ue|^{2\beta} \, \xe_1\ue \, dx + ...\\
	\leq \frac{c}{\eps} \int_{\Omega} \eta^{2\beta+2} \, \we^{\frac{p-1}{2}} \,	|\ve\ue|^{2\beta} \, |\nae\ve\ue| \, dx 
+ c \int_{\Omega} \eta^{2\beta+2} \, \we^{\frac{p}{2}} \,	|\ve\ue|^{2\beta} \, dx \, . 
\end{multline*}	
Following now the estimates from $(L_2)$ we get that
\begin{multline*}
	(R_1) \leq \frac{l}{100} \int_{\Omega} \eta^{2\beta+2} \, \we^{\frac{p-2}{2}} \, |\ve\ue|^{2\beta} \, |\nae\nae\ue|^2 \, dx\\
+c (\beta+1)^2 ||\nae\eta||^2_{L^{\infty}} \, \eps^2 \int_{\Omega} \eta^{2\beta} \, \we^{\frac{p}{2}} \, |\ve\ue|^{2\beta-2} \, |\nae\nae\ue|^2 \, dx .
\end{multline*}
We can finish now the proof by combining the above estimates.	
\end{proof}

Using the fact that $|\vv\ue|^2 \leq 2 |\nae \nae \ue|^2$, we can adapt the proof of Lemma \ref{lemma:vertical3} to the case $\beta=0$, to obtain the following estimate.

\begin{cor}\label{cor:vertical3}
Let $0< \delta <1$ and $\eta \in C_0^{\infty} (\Omega )$ be such that $0 \leq \eta \leq 1$.
Then there exists a constant $c>0$ depending only on $p$, $l$ and $L$ such that for any solution $\ue \in C^{\infty} (\Omega )$ of \eqref{eq:PDE2} we have
\begin{multline}\label{ineq:vertical2.1}
\int_{\Omega} \eta^2 \, \we^{\frac{p-2}{2}} \,
|\nae \nae \ue |^2 \, dx \\
\leq c \left( 1 + ||\nae \eta||^2_{L^{\infty}} + ||\vv\eta||_{L^{\infty}} \right)\int_{\supp (\eta)}	\we^{\frac{p}{2}} \, dx .
\end{multline}
\end{cor} 

\begin{lemma}\label{lemma:vertical4}
Let $0< \delta <1$, $\beta \geq 1$ and $\eta \in C_0^{\infty} (\Omega )$ be such that $0 \leq \eta \leq 1$.
Then there exists a constant $c>0$ depending only on $p$, $l$ and $L$ such that for any solution $\ue \in C^{\infty} (\Omega )$ of \eqref{eq:PDE2} we have
\begin{multline}\label{ineq:vertical4}
\int_{\Omega} \eta^{2\beta+2} \, \we^{\frac{p-2}{2}} \, |\vv\ue|^{2\beta} \,
|\nae \nae \ue |^2 \, dx \\
\leq c^{\beta} (\beta +1)^{4\beta} \| \nae\eta \|^{2\beta}_{L^{\infty}} \int_{\Omega} \eta^{2} \,  \we^{\frac{p-2}{2}+\beta} \, |\nae\nae\ue|^2 dx \, .\\
 \end{multline}	
\end{lemma}

\begin{proof}
The case $\beta=1$ is included in Lemma (\ref{lemma:vertical3}).\\
In the case of $\beta>1$, in the right hand side of  \eqref{ineq:vertical3} we use Young's inequality with the constants $\overline{p}=\frac{\beta}{\beta-1}$ and $\overline{q} = \beta$.	
\begin{multline*}
	\int_{\Omega} \eta^{2\beta+2} \, \we^{\frac{p-2}{2}} \, |\ve\ue|^{2\beta} \,
|\nae \nae \ue |^2 \, dx \\
\leq c \eps^2 (\beta +1)^4 \| \nae\eta \|^2_{L^{\infty}} \int_{\Omega} \eta^{2\beta} \,  \we^{\frac{p}{2}} \, |\ve \ue |^{2\beta-2} \, |\nae\nae\ue|^2 dx\\
\shoveleft{\leq \left( \int_{\Omega} \eta^{2\beta+2} \, \we^{\frac{p-2}{2}} \, |\ve\ue|^{2\beta} \, |\nae\nae\ue|^2 \, dx \right)^{\frac{\beta-1}{\beta}} }\\
\cdot \left( c^{\beta} \, \eps^{2\beta} \, (\beta+1)^{4\beta} \, 
||\nae\eta||^{2\beta}_{L^{\infty}} \, \int_{\Omega} \eta^{2}\, \we^{\frac{p-2}{2}+\beta} \, |\nae\nae\ue|^2 \, dx \right)^{\frac{1}{\beta}} \, .
\end{multline*}
Taking into consideration a division by $\eps^{2\beta}$, estimate \eqref{ineq:vertical4} is now a simple consequence of the above inequality.
\end{proof}

{\bf Proof of Theorem \ref{thm:Lipschitz}.}
\begin{proof}
We start with the first term on the right hand side of \eqref{ineq:vertical2}.
By Young's inequality with exponents $\overline{p} = \beta+1$ and $\overline{q} = \frac{\beta+1}{\beta}$ and Lemma \ref{lemma:vertical4} we get that
\begin{multline*}
c (\beta+1)^4 \int_{\Omega} \eta^2 \, \we^{\frac{p-2}{2}+\beta} \, |\vv\ue|^2 \, dx\\
\shoveleft{\leq c (\beta+1)^4 \left( \int_{\Omega} \eta^{2\beta+2} \, \we^{\frac{p-2}{2}} \, |\vv\ue|^{2\beta+2} \, dx \right)^{\frac{1}{\beta+1}}}\\ 
\shoveright{\cdot \left( \int_{\supp \eta} \we^{\frac{p}{2}+\beta} \, dx \right)^{\frac{\beta}{\beta+1}}}\\
\shoveleft{\leq c (\beta+1)^4 \left( 2\, \int_{\Omega} \eta^{2\beta+2} \, \we^{\frac{p-2}{2}} \, |\vv\ue|^{2\beta} \, |\nae\nae\ue|^2 \, dx \right)^{\frac{1}{\beta+1}}}\\ 
\shoveright{\cdot \left( \int_{\supp \eta} \we^{\frac{p}{2}+\beta} \, dx \right)^{\frac{\beta}{\beta+1}}}\\
\shoveleft{\leq c (\beta+1)^4 \left( c^{\beta} (\beta+1)^{4\beta} \, ||\nae\eta||^{2\beta}_{L^{\infty}}\, \int_{\Omega} \eta^{2} \, \we^{\frac{p-2}{2}+\beta} \, |\nae\nae\ue|^{2} \, dx \right)^{\frac{1}{\beta+1}}}\\ 
\shoveright{\cdot \left( \int_{\supp \eta} \we^{\frac{p}{2}+\beta} \, dx \right)^{\frac{\beta}{\beta+1}}}\\
\shoveleft{\leq c^{\frac{2\beta+1}{\beta+1}} (\beta+1)^{\frac{8\beta+4}{\beta+1}} \, ||\nae\eta||^{\frac{2\beta}{\beta+1}}_{L^{\infty}}\, \left( \int_{\Omega} \eta^{2} \, \we^{\frac{p-2}{2}+\beta} \, |\nae\nae\ue|^{2} \, dx \right)^{\frac{1}{\beta+1}}}\\ 
\shoveright{\cdot \left( \int_{\supp \eta} \we^{\frac{p}{2}+\beta} \, dx \right)^{\frac{\beta}{\beta+1}}}\\
\shoveleft{\leq \frac{1}{\beta+1} \,\int_{\Omega} \eta^{2} \, \we^{\frac{p-2}{2}+\beta} \, |\nae\nae\ue|^{2} \, dx}\\ 
+ \frac{\beta}{\beta+1} \, c^{\frac{2\beta+1}{\beta}} (\beta+1)^{\frac{8\beta+4}{\beta}} \, ||\nae\eta||^{2}_{L^{\infty}}\, \int_{\supp \eta} \we^{\frac{p}{2}+\beta} \, dx \,.
\end{multline*}  

Hence, inequality \eqref{ineq:vertical2} implies the following estimate:
\begin{multline}\label{ineq:vertical5}
\int_{\Omega} \eta^2 \, \we^{\frac{p-2}{2}+\beta}\, |\nae\nae\ue|^2 \, dx\\
\leq c (\beta+1)^{12} \, \left(1+ ||\nae\eta||^2_{L^{\infty}} + ||\vv\eta||_{L^{\infty}} \right) \, \int_{\supp \eta}	 \we^{\frac{p}{2}+\beta} \, dx.
\end{multline}
 
Since for any integer $1 \leq i \leq 8$ we have
\begin{equation*}
	\left( \xe_i \left( \eta \, \we^{\frac{p}{4}+\frac{\beta}{2}} \right) \right)^2 \\
	\leq 2 |\nae\eta|^2 \, \we^{\frac{p}{2} +\beta} \, +  2 \left( \frac{p}{2}+\beta \right)^2 \, \eta^2 \, \we^{\frac{p}{2}+\beta} \, |\nae\nae\ue|^2 \, ,
\end{equation*}
it follows that 
\begin{multline*}
	\int_{\Omega} \left| \nae \left( \eta \, \we^{\frac{p}{4} + \frac{\beta}{2}} \right) \right|^2 \, dx \\
	\leq c (\beta+ 1)^{14} \, \left( 1 + ||\nae\eta||^2_{L^{\infty}} + ||\vv\eta||_{L^{\infty}} \right) \, \int_{\supp \eta}	 \we^{\frac{p}{2}+\beta} \, dx.
\end{multline*}

Similarly to \eqref{def:distance}, for any small $\eps >0$, consider subunitary curves associated to $\{ \xe_i , 1 \leq i \leq 8 \}$, the control distance $d^\eps$ and balls $B_r^{\eps}$. \\
Notice that for all $\eps >0$ and  $x,y \in \SU(3)$ we have $d^{\eps} (x,y) \leq d(x,y)$, and hence it follows that $B_r \subset B_r^{\eps}$. The homogeneous dimension $Q=10$ provides a constant $c$ independent of $\eps$ such that for volumes of balls of radius $0 < r \leq 1$ we have
$$ c r^Q \leq |B_r| \leq |B_r^{\eps}| \, .$$
By  \cite[Theorem V.4.5, page 70]{VSC92}, the Sobolev inequality holds for $\kappa=\frac{Q}{Q-2} = \frac{5}{4}$ and a constant $c$, depending only on $Q$ and independent of $\eps$. For a careful study of the independence of $c$ of $\eps$, see \cite{CC16}. Therefore, for $0 < \frac{r}{2} \leq r_1 < r_2 \leq r$ and appropriate cut-off function $\eta$ we have
\begin{multline*}
	\left[ \int_{B_{r_1}^{\eps}} \eta^{2\kappa} \, \we^{(\frac{p}{2}+\beta) \kappa} \right]^{\frac{1}{\kappa}}\\
	\leq c (\beta+ 1)^{14} \, \left( 1 + ||\nae\eta||^2_{L^{\infty}} + ||\vv\eta||_{L^{\infty}} \right) \, \int_{B_{r_2}^{\eps}}	 \we^{\frac{p}{2}+\beta} \, dx.
\end{multline*}

The well-known Moser iteration leads to a constant independent of $\eps$, such that for any weak solution $\ue$ of \eqref{eq:PDE2} in $B$, satisfying $\ue = u$ on $\partial B$ we have
\begin{equation}\label{eq:Linfinityeps}
\sup_{B_{r/2}^{\eps}} |\nae \ue | \leq c \left( -\hspace{-0.45cm}\int_{B_r^{\eps}} (\delta + |\nae\ue |^2 )^{\frac{p}{2}} dx \right)^{\frac{1}{p}} \, .
 	\end{equation}
Letting $\eps \to 0$ in \eqref{eq:Linfinityeps}, we obtain \eqref{ineq:Linfinity}. 	
\end{proof}

\section{The proof of Theorem \ref{thm:GradientHolder}}

Based on the Lipschitz regularity from Theorem \ref{thm:Lipschitz} and \cite[Theorem 1.1]{DM09} we have the following result:

\begin{theorem} \label{thm:smooth}
Let $p \geq 2$, $\delta >0$ and $u \in  W^{1,p}_{\Ho\loc} (\Omega )$ be a weak solution of \eqref{eq:PDE1}. Then $u \in C^{\infty} (\Omega)$.
\end{theorem}

We can observe that the estimates from the Lemmas and Corollaries from the previous section are homogeneous in $\eps$. Therefore, by dividing with the corresponding power of $\eps$ and then letting $\eps \to 0$, we obtain the following intrinsic Cacciopoli type inequalities for solutions of \eqref{eq:PDE1}. Similar inequalities were obtained in the case of the Heisenberg group in \cite{MM07,MZGZ09,R18,Z17}. We will use the notation $w = \delta + |\he u|^2$.

\begin{cor}\label{cor:Cacciopoli1}
Let $0< \delta <1$ and $\eta \in C_0^{\infty} (\Omega )$ be such that $0\leq \eta \leq 1$.
Then there exists a constant $c>0$ depending only on $p$, $l$ and $L$ such that for any solution $u \in C^{\infty}(\Omega)$ of \eqref{eq:PDE1} the following inequalities hold:

(1) If $\beta \geq 0$, then
\begin{multline}\label{ineq:Cacciopoli1}
\int_{\Omega} \eta^2 \, w^{\frac{p-2}{2}} \, |\vv u|^{2\beta} \,
|\he \vv u |^2 \, dx 
\leq c \int_{\Omega} |\he \eta|^2 \,  w^{\frac{p-2}{2}} \, |\vv u |^{2\beta+2} \, dx\\
 + c (\beta + 1)^2 \int_{\Omega}	\eta^2 \, w^{\frac{p}{2}} \, |\vv u|^{2\beta} \, dx .
\end{multline}

(2) It $\beta \geq 0$, then
\begin{multline}\label{ineq:Cacciopoli2}
\int_{\Omega} \eta^2 \, w^{\frac{p-2}{2}+\beta} \,
|\he \he u |^2 \, dx 
\leq c (\beta +1)^4 \int_{\Omega} \eta^2 \,  w^{\frac{p-2}{2}+\beta} \, |\vv u |^{2} \, dx\\
 + c (\beta+1)^2 \int_{\Omega}	(\eta^2+|\he \eta|^2 + \eta|\vv\eta| )\, w^{\frac{p}{2}+\beta} \, dx .
\end{multline}

(3) It $\beta \geq 1$, then
\begin{multline}\label{ineq:Cacciopoli3}
\int_{\Omega} \eta^{2\beta+2} \, w^{\frac{p-2}{2}} \, |\vv u|^{2\beta} \,
|\he \he u |^2 \, dx \\
\leq c (\beta +1)^4 \| \he\eta \|^2_{L^{\infty}} \int_{\Omega} \eta^{2\beta} \,  w^{\frac{p}{2}} \, |\vv u |^{2\beta-2} \, |\he\he u|^2 dx \, .\\
 \end{multline}

(4) If $\beta \geq 1$, then
\begin{multline}\label{ineq:Cacciopoli4}
\int_{\Omega} \eta^{2\beta+2} \, w^{\frac{p-2}{2}} \, |\vv u|^{2\beta} \,
|\he \he u |^2 \, dx \\
\leq c^{\beta} (\beta +1)^{4\beta} \| \he \eta \|^{2\beta}_{L^{\infty}} \int_{\Omega} \eta^{2} \,  w^{\frac{p-2}{2}+\beta} \, |\he \he u|^2 dx \, .\\
 \end{multline}	

(5) If $\beta \geq 0$, then
\begin{multline}\label{ineq:Cacciopoli5}
\int_{\Omega} \eta^2 \, w^{\frac{p-2}{2}+\beta}\, |\he \he u|^2 \, dx\\
\leq c (\beta+1)^{12} \, \left(1+ ||\he \eta||^2_{L^{\infty}} + ||\vv \eta||_{L^{\infty}} \right) \, \int_{\supp \eta}	 w^{\frac{p}{2}+\beta} \, dx.
\end{multline}
\end{cor}

In case of $\delta =0$ the key result in proving the $C^{1,\alpha}$ regularity of weak solutions of equation \eqref{eq:PDE1}  is the following lemma:

\begin{lemma} \label{lemma:DeGiorgi}
Let $\delta >0$, $u \in C^{\infty} (\Omega)$ be a solution of \eqref{eq:PDE1} and consider a CC-ball $B_{3r_0} \subset \Omega$.
For any $q  \geq 4$ there exists a constant $c>0$, depending only on $\G, p, l, L, r_0$ and $q$, such that for all $k \in {\mathbb R}$, $|k| < M$, $0 <r' < r < r_0$, $s \in I$  we have  
\begin{multline} \label{ineq:superCacciopoli}
\int_{A^+_{s,k,r'}} \, (\delta + |\he u|^{2})^{\frac{p-2}{2}} \, |\he (X_s u - k)^+|^2 \, dx \\
\leq  \frac{c}{(r-r')^2} \, \int_{A^+_{s,k,r}} \, (\delta + |\he u|^2)^{\frac{p-2}{2}} \, ((X_s u - k)^+)^2 \, dx \\
+ c (\delta + M^2)^{\frac{p}{2}} \, |A^+_{s,k,r}|^{1-\frac{2}{q}} \, ,
\end{multline}
where  $M = \sup_{B_{2r_0}} |\vv u|$ and  $A^+_{s,k,r} = \{ x \in B_r \, : \, X_s u (x) - k > 0\}$.
\end{lemma}

\begin{proof} 
We will present the proof for $s=1$, the other cases are identical. Let us denote $v = (X_1u - k)^+$. As in Section \S 3, let us differentiate equation \eqref{eq:PDE1} with respect to $X_1$, multiply it by a $\phi \in C_0^{\infty} (\Omega)$ and integrate. In this way we obtain
\begin{multline}\label{eq:horizontal1}
\sum_{i,j} \int_{\Omega} a_{ij} \, X_j X_1 u \, X_i \phi \, dx 
= -\int_{\Omega} X_7 a_2 \, \phi \, dx -  \int_{\Omega} a_3 \, X_5 \phi \, dx + ...\\
+ \sum_i \int_{\Omega} a_{i2} \, X_7 u \, X_i \phi \, dx - \sum_i \int_{\Omega} a_{i3} \, X_5 u \, X_i \phi \, dx + ... 	
\end{multline}
 Consider a cut-off function $\eta \in C_0^{\infty} (B_r )$ such that $\eta \equiv 1$ in $B_{r'}$, $||\he \eta ||_{L^{\infty}} \leq \frac{2}{r-r'}$ and 
 $||\vv \eta ||_{L^{\infty}} \leq \frac{8}{(r-r')^2}$ . After substituting the test function $\phi = \eta^2 v$ in equation \eqref{eq:horizontal1}, we get the following terms:
 $$L_1 + L_2 = R_1 + R_2 + ... + R_3 + R_4  + ...$$  
We  will estimate each term. Note that $X_j X_1 u (x) = X_j v (x)$ if $v(x) \neq 0$ and we can assume $|B_{2r_0}| \leq 1$.
\begin{multline*} 
\shoveright{L_1 = \sum_{i,j} \int_{B_r} a_{ij} \, X_j X_1 u \, \eta^2 \, X_i v \, dx 
\geq l \int_{B_r} \eta^2 \, w^{\frac{p-2}{2}} \, | \he v |^2 \, dx \, .}	  
\end{multline*}

\begin{multline*}
L_2 = \sum_{i,j} \int_{B_r} a_{ij} \, X_j X_1 u \, 2 \eta \, X_i \eta \, v dx
\leq c \int_{B_r} w^{\frac{p-2}{2}} \, |\he v| \, \eta \, |\he \eta | \, v \, dx \\
\leq \frac{l}{100} \int_{B_r} \eta^2 \, w^{\frac{p-2}{2}} \, |\he v|^2 \, dx + c \int_{B_r} |\he \eta|^2 \, w^{\frac{p-2}{2}} \, v^2 \, dx \, . 
\end{multline*}

\begin{multline*}
R_1 = -\int_{B_r} X_7 a_2 \, \eta^2 \, v \, dx =- \sum_i \int_{B_r} a_{2i} \, X_7 X_i u \, \eta^2 \, v \, dx	\\
\shoveleft{= -\sum_i \int_{B_r} a_{2i} \, X_i X_7 u \, \eta^2 \, v \, dx +4 \int_{B_r} a_{21} \, X_2 u \, \eta^2 \, v \, dx + ...}\\
\shoveleft{\leq c \int_{B_r} \eta^2 \, w^{\frac{p-2}{2}} \, |\he \vv u| \, v \, dx + c \int_{B_r} \eta^2 \, w^{\frac{p-1}{2}} \, v \, dx}\\
\shoveleft{\leq c \int_{B_r} \eta^2 \, w^{\frac{p-2}{2}} \, |\he \vv u| \, v \, dx + c \int_{B_r} \eta^2 \, w^{\frac{p-2}{2}} \, v^2 \, dx + c \int_{A^+_{1,k,r}} \eta^2 \, w^{\frac{p}{2}} \, dx}\\
\shoveleft{\leq c \int_{B_r} \eta^2 \, w^{\frac{p-2}{2}} \, |\he \vv u| \, v \, dx + c \int_{B_r} \eta^2 \, w^{\frac{p-2}{2}} \, v^2 \, dx}\\ 
\shoveright{+ c (\delta+M^2 )^{\frac{p}{2}} \, | A^+_{1,k,r} |}\\
\shoveleft{\leq c \int_{B_r} \eta^2 \, w^{\frac{p-2}{2}} \, |\he \vv u| \, v \, dx + c \int_{B_r} \eta^2 \, w^{\frac{p-2}{2}} \, v^2 \, dx }\\
+ c (\delta+M^2 )^{\frac{p}{2}} \, | A^+_{1,k,r} |^{1-\frac{2}{q}} \, .
\end{multline*}

\begin{multline*}
R_2 = \frac{1}{2} \int_{B_r} a_3 \, (\eta^2 X_5 v + 2\eta \, X_5 \eta \, v )\, dx \\
\shoveleft{\leq c \int_{B_r} w^{\frac{p-1}{2}} \, \eta^2 \, |\he v| \, dx + c \int_{B_r} w^{\frac{p-1}{2}} \, \eta \, |\he \eta| \, v \, dx }\\
\shoveleft{\leq \frac{l}{100} \int_{B_r} \eta^2 \, w^{\frac{p-2}{2}} \, |\he v|^2 \, dx + c\int_{A^+_{1,k,r}} \eta^2 \, w^{\frac{p}{2}} \, dx}\\ 
\shoveright{+ c \int_{B_r} | \he \eta |^2 \, w^{\frac{p-2}{2}} \,  v^2 \, dx }\\
\shoveleft{\leq \frac{l}{100} \int_{B_r} \eta^2 \, w^{\frac{p-2}{2}} \, |\he v|^2 \, dx + c \int_{B_r} |\he \eta|^2 \, w^{\frac{p-2}{2}} \,  v^2 \, dx} \\ +
c (\delta+M^2 )^{\frac{p}{2}} \, | A^+_{1,k,r} |^{1-\frac{2}{q}} \, .
\end{multline*}
\begin{multline*}
R_3 =  \sum_i \int_{\Omega} a_{i2} \, X_7 u \, (\eta^2 X_i v + 2\eta \, X_i \eta \, v ) \, dx \\
\shoveleft{\leq c \int_{B_r} w^{\frac{p-2}{2}} \, |\vv u| \,  \eta^2 \, |\he v| \, dx + c \int_{B_r} w^{\frac{p-2}{2}} \, |\vv u| \, \eta \, |\he \eta| \, v \, dx} \\
\shoveleft{\leq \frac{l}{100} \int_{B_r} \eta^2 \, w^{\frac{p-2}{2}} \, |\he v|^2 \, dx + c \int_{B_r} | \he \eta |^2 \, w^{\frac{p-2}{2}} \,  v^2 \, dx }\\
+ 	c\int_{B_r} \eta^2 \, w^{\frac{p-2}{2}} \, |\vv u|^2 \, dx \, .
\end{multline*}

The last term needs more attention. We will use the H\"{o}lder inequality and inequalities \eqref{ineq:Cacciopoli4} and \eqref{ineq:Cacciopoli5}. All multipliers involving $q$ and $r_0$ will be included in the general constant $c$.

\begin{multline*}
c\int_{B_r} \eta^2 \, w^{\frac{p-2}{2}} \, |\vv u|^2 \, dx \\
\shoveleft{\leq c\left( \int_{A^+_{1,k,r}} w^{\frac{p-2}{2}} \, dx \right)^{1-\frac{2}{q}} \; 	\left( \int_{B_r} \eta^q \, w^{\frac{p-2}{2}} \, |\vv u|^q \, dx \right)^{\frac{2}{q}}} \\
\shoveleft{\leq c(\delta +M^2 )^{\frac{(p-2)(q-2)}{2q}} \, |A^+_{1,k,r}|^{1-\frac{2}{q}} \, 
\left( \int_{B_r} \eta^q \, w^{\frac{p-2}{2}} \, |\vv u|^{q-2} \, |\he\he u|^2 \, dx \right)^{\frac{2}{q}} }\\
\shoveleft{\leq c(\delta +M^2 )^{\frac{(p-2)(q-2)}{2q}} \, |A^+_{1,k,r}|^{1-\frac{2}{q}}} \\ 
\cdot \left( c^{\frac{q-2}{2}} \, \left( \frac{q}{2} \right)^{2q-4} \, \left( \frac{8}{r_0} \right)^{q-2} \, \int_{B_{\frac{5r_0}{4}}}  \, w^{\frac{p+q-4}{2}} \, |\he\he u|^2 \, dx \right)^{\frac{2}{q}}\\
\shoveleft{\leq c (\delta +M^2 )^{\frac{(p-2)(q-2)}{2q}} \, |A^+_{1,k,r}|^{1-\frac{2}{q}}} 
\cdot \left( \int_{B_{\frac{6r_0}{4}}}  \, w^{\frac{p+q-2}{2}} \, dx \right)^{\frac{2}{q}}\\
\shoveleft{\leq c (\delta +M^2 )^{\frac{(p-2)(q-2)}{2q}} \, |A^+_{1,k,r}|^{1-\frac{2}{q}}} 
\cdot \left( \delta+M^2 \right)^{\frac{p+q-2}{q}}\\
\shoveleft{\leq c (\delta +M^2 )^{\frac{p}{2}} \, |A^+_{1,k,r}|^{1-\frac{2}{q}}} \, .\\
\end{multline*}
The estimate of $R_4$ is similar to the estimate of $R_2$. In conclusion, at this stage for a constant $c_0>0$, we have the following estimate:
\begin{multline}\label{ineq:FirstEstimate}
\int_{B_r} \eta^2 \, w^{\frac{p-2}{2}} \, |\he v|^2 \, dx \leq
c_0 \int_{B_r} \eta^2 \, w^{\frac{p-2}{2}} \, |\he \vv u| \, v \, dx \\ 	
+ c_0 \int_{B_r} (\eta^2 + |\he \eta|^2 )\, w^{\frac{p-2}{2}} \, v^2 \, dx + 
c_0 \, (\delta+M^2)^{\frac{p}{2}} \, |A_{1,k,r}^+ |^{1-\frac{2}{q}} \, .
\end{multline}
It is left to estimate the term
$$A_0 = c_0 \int_{B_r} \eta^2 \, w^{\frac{p-2}{2}} \, |\he\vv u| \, v \ dx \, .$$
By introducing the term
$$\kappa = \left( \int_{B_r} (\eta^2 + |\he \eta|^2 ) \, w^{\frac{p-2}{2}} \, v^2 \, dx + \int_{B_r} \eta^2 \, w^{\frac{p-2}{2}} \, |\he v |^2 \, dx \right)^{\frac{1}{2}} \, ,$$
inequality \eqref{ineq:FirstEstimate} can be rewritten as 
\begin{multline}\label{ineq:kappa}
\kappa^2 \leq A_0 + (c_0 +1) \int_{B_r} (\eta^2 + |\he \eta|^2 )\, w^{\frac{p-2}{2}} \, v^2 \, dx \\
+ c_0 \, (\delta+M^2)^{\frac{p}{2}} \, |A_{1,k,r}^+ |^{1-\frac{2}{q}} \, .
\end{multline}
We will focus now on $A_0$. By H\"{o}lder's inequality we obtain
\begin{multline*}
A_0 \leq c_0  \left( \int_{A^+_{1,k,r}} \eta^2 \, w^{\frac{p-2}{2}} \, dx \right)^{\frac{1}{2}} \; 	\left( \int_{B_r} \eta^2 \, w^{\frac{p-2}{2}} \, |\he \vv u|^2 \, v^2 \, dx \right)^{\frac{1}{2}}	 \\
\leq c_0  (\delta + M^2)^{\frac{p-2}{4}} \, |A^+_{1,k,r}|^{\frac{1}{2}} \; 	\left( \int_{B_r} \eta^2 \, w^{\frac{p-2}{2}} \, |\he \vv u|^2 \, v^2 \, dx \right)^{\frac{1}{2}}\, .	 
\end{multline*}
For $\beta \geq 0$, we introduce the following terms:
\begin{align*}
& \Gamma_{\beta} = \int_{B_r} \eta^2 \, w^{\frac{p-2}{2}} \, |\he \vv u|^2 \, |\vv u|^{\beta} \, v^2 \, dx \, , \\
& \Lambda_{\beta} = \int_{B_r} \eta^2 \, w^{\frac{p}{2}} \, |\vv u|^{\beta} \, v^2 \, dx \, .
\end{align*}
Note that we have
\begin{equation}\label{ineq:A0}
	A_0 \leq c_0  (\delta + M^2)^{\frac{p-2}{4}} \, |A^+_{1,k,r}|^{\frac{1}{2}} \, \Gamma_0^{\frac{1}{2}} \, .
\end{equation}

By the fact that $v^2 \leq 4 (\delta+M^2)$ and after repeated use of the inequalities \eqref{ineq:Cacciopoli1}-\eqref{ineq:Cacciopoli5}, we find a constant $c>0$ depending on $p, l, L, r_0, \beta$ such that
\begin{equation}\label{ineq:GammaLambda}
\Gamma_\beta + \Lambda_\beta \leq c (\delta+M^2)^{\frac{p+\beta+2}{2}} \, .	
\end{equation}
Applying the H\"{o}lder inequality to $\Lambda_\beta$, for $\beta>0$ we get that
\begin{equation}\label{ineq:LambdaBeta}
\Lambda_\beta \leq c \kappa \, (\delta+M^2 )^{\frac{1}{2}}	\, \Lambda_{2\beta}^{\frac{1}{2}} \, ,
\end{equation}
and after iterating \eqref{ineq:LambdaBeta} $m$ times, we find that there exists a constant $c>0$ depending also on $p, l, L, r_0, \beta$ and $m$ such that
 \begin{equation}\label{ineq:LambdaBetam}
\Lambda_\beta \leq c \kappa^{2 - \frac{1}{2^{m-1}}} \, (\delta+M^2 )^{1- \frac{1}{2^m}}	\, \Lambda_{2^m\beta}^{\frac{1}{2^m}} \, ,
\end{equation}

To estimate $\Gamma_\beta$ let us differentiate \eqref{eq:PDE1} with respect to $X_7$ to get
\begin{multline*}
	\sum_{i,j =1}^6 \int_{\Omega} a_{ij} \, X_j X_7 u \, X_i \phi \, dx \\
	= 4 \int_{\Omega} a_1 \, X_2 \phi \, dx + ... +4 \sum_{i=1}^6 \int_{\Omega} a_{i1} X_2 u \, X_i \phi \, dx - ... 
\end{multline*}
We will use $\phi = \eta^2 \, v^2 \, |\vv u|^{\beta} \, X_7 u$. In $X_i \phi$ we will order the four terms in the following way:
\begin{multline*}
	X_i \phi = 2 \eta \, X_i \eta \, v^2 \, |\vv u|^\beta \, X_7 u + \eta^2 \, 2v \, X_i v \, |\vv u|^\beta \, X_7 u\\
	+ \eta^2 \, v^2 \, \frac{\beta}{2} |\vv u|^{\beta-2} \, X_i (|\vv u|^2 ) \, X_7 u + \eta^2 \, v^2 \, |\vv u|^\beta \, X_i X_7 u \, .
\end{multline*}
By repeating the same steps for $X_8$ and adding the two equations we get the following terms:
\begin{equation*}
	L_1 + L_2 + L_3 + L_4 = \sum_{i=1}^4 R_{1i} + ... +  \sum_{i=1}^4 R_{2i} + ...
\end{equation*}
For each term we have the following estimates.
\begin{multline*}
L_1 \leq c \int_{\Omega} w^{\frac{p-2}{2}} \, |\he \vv u| \, \eta \, |\he \eta| \, v^2 \, |\vv u|^{\beta+1} \, dx\\
\shoveleft{\leq c \left( \int_{\Omega} |\he \eta|^2 w^{\frac{p-2}{2}} \, v^2 \, dx 	 \right)^{\frac{1}{2}} }\\
\cdot \left( \int_{\Omega} \eta^2 \, w^{\frac{p-2}{2}} \, |\he\vv u|^2 \, |\vv u|^{2\beta+2} \, \, v^2 \, dx \right)^{\frac{1}{2}} \leq c \, \kappa \, \Gamma_{2\beta+2}^{\frac{1}{2}} \, . 
\end{multline*}

\begin{multline*}
\shoveright{L_2 \leq \text{similarly to $L_1$} \leq c \,  \kappa \, \Gamma_{2\beta+2}^{\frac{1}{2}} \, .}
\end{multline*}

\begin{multline*}
\shoveright{L_3 \geq \frac{l \beta}{4} \int_{\Omega} \eta^2 \, w^{\frac{p-2}{2}} \, |\he (|\vv u |^2 )|^2 \, |\vv u|^{\beta-2} \, v^2 \, dx \, .	}
\end{multline*}

\begin{multline*}
\shoveright{L_4 \geq l \int_{\Omega} \eta^2 \, w^{\frac{p-2}{2}} \, |\he \vv u |^2 \, |\vv u|^{\beta} \, v^2 \, dx \, 	= l \Gamma_\beta \, .}
\end{multline*}

By H\"{o}lder's inequality we get 
\begin{multline*}
\shoveright{R_{11} + R_{12} + R_{21} + R_{22} \leq c \, \kappa \, \Lambda_{2\beta+2}^{\frac{1}{2}} \, .	}
\end{multline*}

Young's inequality leads to
\begin{multline*}
\shoveright{R_{13} + R_{14} + R_{23} + R_{24} \leq  \frac{l}{100} \Gamma_{\beta} + c \, \kappa \, (\beta +1)^2 \, (\delta + M^2)^{\frac{1}{2}} \, \Lambda_{2\beta}^{\frac{1}{2}} \, .	}	
\end{multline*}

Therefore in case of $\beta \geq 2$ we obtained the following inequality:
\begin{equation}\label{ineq:GammaBeta}
\Gamma_{\beta} \leq c \, \kappa \, \left( \Gamma_{2\beta+2}^{\frac{1}{2}} + 	\Lambda_{2\beta+2}^{\frac{1}{2}} + (\delta + M^2)^{\frac{1}{2}} \, \Lambda_{2\beta}^{\frac{1}{2}} \right) \, ,
\end{equation}
where the constant $c$ depends on $\beta$.\\
In case of $\beta =0$ the terms $L_3$, $R_{13}$ and $R_{23}$ are missing and, for an integer $m \in {\mathbb N}$,  the estimate for $R_{14} + R_{24}$ can be changed to the following.
\begin{multline*}
R_{14} + R_{24} \leq c \int_{\Omega} w^{\frac{p-1}{2}} \, \eta^2 \, v^2 \, |\he\vv u| \, dx\\
\shoveleft{\leq \frac{l}{100} \int_{\Omega} \eta^2 \, w^{\frac{p-2}{2}} \, |\he\vv u|^2 \, v^2 \, dx + c \int_{\Omega} \eta^2 \, w^{\frac{p}{2}} \, v^2 \, dx }\\
\shoveleft{ \leq \frac{l}{100} \Gamma_0 +  	c (\delta+M^2 ) \int_{\Omega} \eta^2 \, w^{\frac{p-2}{2}} \, v^2 \, dx }\\
\shoveleft{ \leq \frac{l}{100} \Gamma_0 +  	c (\delta+M^2 ) \kappa^{2(1-\frac{1}{2^{m+1}})} \; \left( \int_{\Omega} \eta^2 \, w^{\frac{p-2}{2}} \, v^2 \, dx \right)^{\frac{1}{2^{m+1}}}}\\
\shoveleft{\leq \frac{l}{100} \Gamma_0 +  	c \, \kappa^{2-\frac{1}{2^m}} \, (\delta+M^2 )^{1+\frac{p}{2^{m+2}}} \, .}\\
\end{multline*}
Therefore, we have obtained the following estimate:
\begin{equation}\label{ineq:GammaZero}
\Gamma_0 \leq c \kappa \left( \Gamma_2^{\frac{1}{2}} + \Lambda_2^{\frac{1}{2}} + (\delta+M^2)^{1+ \frac{p}{2^{m+2}}} \, \kappa^{1-\frac{1}{2^m}} \right) \, .	
\end{equation}
In inequality \eqref{ineq:GammaZero} we will have to have to iteratively apply \eqref{ineq:GammaBeta}. First, by using \eqref{ineq:LambdaBeta} and \eqref{ineq:GammaLambda}, we can rewrite \eqref{ineq:GammaBeta} in the following way:
\begin{equation}\label{ineq:GammaBeta2}
\Gamma_\beta \leq c \, \kappa \, \Gamma_{2\beta+2}^{\frac{1}{2}} + c \, \kappa^{2 - \frac{1}{2^m}} \, (\delta+M^2)^{\frac{\beta+2}{2}+\frac{p}{2^{m+2}}} \, .	
\end{equation}
After $m$ iterations of \eqref{ineq:GammaBeta2} and by choosing  $\beta_m = 2^m-2$, we get the following inequality:
\begin{multline}\label{ineq:Gamma2=GammaBeta2}
\Gamma_2 = \Gamma_{\beta_2} \leq c \, \kappa^{\sum_{i=0}^{m-1} \frac{1}{2^i}} \; \Gamma_{\beta_{2+m}}^{\frac{1}{2^m}} +
\sum_{i=0}^{m-1} \kappa^{2-\frac{1}{2^{m+i}}} \; (\delta+M^2 )^{2 + \frac{p}{2^{m+2+i}}} \, .	
\end{multline}
By applying \eqref{ineq:GammaLambda} and \eqref{ineq:Gamma2=GammaBeta2} in \eqref{ineq:GammaZero} we get that
\begin{multline*}
\Gamma_0 \leq c \, \kappa^{2-\frac{1}{2^m}} \; (\delta+M^2)^{1+\frac{p}{2^{m+2}}} + c \, \sum_{i=1}^m \kappa^{2 - \frac{1}{2^{m+i}}} \; (\delta+M^2)^{1+\frac{p}{2^{m+2+i}}}\\
+ 2 c \, \kappa^{2 - \frac{1}{2^{m}}} \; (\delta+M^2)^{1+\frac{p}{2^{m+2}}}	\, .
\end{multline*}
Hence, we obtained a constant $c_1$ such that
\begin{equation}\label{ineq:GammaZerom}
\Gamma_0 \leq c_1 \, \sum_{i=0}^m (\kappa^2)^{1 - \frac{1}{2^{m+i+1}}} \; (\delta+M^2)^{1+\frac{p}{2^{m+i+2}}}	 \, .
\end{equation}
We return now to inequality \eqref{ineq:A0} and obtain
\begin{multline*}
A_0 \leq c_0 (\delta+M^2)^{\frac	{p-2}{4}} \, |A_{1,k,r}^+|^{\frac{1}{2}} \\\cdot \left( c_1 \, \sum_{i=0}^m (\kappa^2)^{1 - \frac{1}{2^{m+i+1}}} \; (\delta+M^2)^{1+\frac{p}{2^{m+i+2}}} \right)^{\frac{1}{2}}\\
\leq c_0 (\delta+M^2)^{\frac{p}{4}} \, |A_{1,k,r}^+|^{\frac{1}{2}} \;
\left( c_1^{\frac{1}{2}} \, \sum_{i=0}^m (\kappa^2)^{\frac{2^{m+i+1}-1}{2^{m+i+2}}} \; (\delta+M^2)^{\frac{p}{2^{m+i+3}}} \right) \\
\leq \sum_{i=0}^m (\kappa^2)^{\frac{2^{m+i+1}-1}{2^{m+i+2}}} \; \left( c_0 c_1^{\frac{1}{2}}\, (\delta+M^2)^{\frac{p(2^{m+i+1}+1)}{2^{m+i+3}}} \, |A_{1,k,r}^+|^{\frac{1}{2}} \right) \, .
\end{multline*}
By applying Young's inequality to each term we obtain
\begin{multline*}
A_0 \leq \sum_{i=0}^m \; \frac{1}{2(m+1)} \kappa^2 	\\ + \sum_{i=0}^m c(m) \left( c_0 \sqrt{c_1} \right)^{\frac{2^{m+i+2}}{2^{m+i+1}+1}} \, (\delta+M^2)^{\frac{p}{2}} \, |A_{1,k,r}^+|^{\frac{2^{m+i+2}}{2^{m+i+2}+2}} \, .
\end{multline*}
By choosing $m \in {\mathbb N}$ such that 
$$1 - \frac{2}{q} \leq \frac{2^{m+2}}{2^{m+2}+2} \, ,$$
and taking into consideration \eqref{ineq:kappa}, we obtain that
\begin{multline*}
A_0 \leq \frac{1}{2} A_0 + 	\frac{c_0 +1}{2} \int_{B_r} (\eta^2 + |\he \eta|^2 )\, w^{\frac{p-2}{2}} \, v^2 \, dx \\
+ (\frac{c_0}{2}+c) \, (\delta+M^2)^{\frac{p}{2}} \, |A_{1,k,r}^+|^{1-\frac{2}{q}} \, . 
\end{multline*}
In conclusion, from \eqref{ineq:FirstEstimate} we get that
\begin{multline*}
\int_{B_r} \eta^2 \, w^{\frac{p-2}{2}} \, |\he v|^2 \, dx \leq
c \int_{B_r} (\eta^2 + |\he \eta|^2 )\, w^{\frac{p-2}{2}} \, v^2 \, dx\\
 + c \, (\delta+M^2)^{\frac{p}{2}} \, |A_{1,k,r}^+ |^{1-\frac{2}{q}} \, ,
\end{multline*}
and this finishes the proof of Lemma \ref{lemma:DeGiorgi}.
\end{proof}

In a similar way we can prove Lemma \ref{lemma:DeGiorgi} for the lower level sets $A_{1,k,r}^-$ and then the proof of Theorem \ref{thm:GradientHolder} relies only on properties of functions belonging to the De Giorgi classes. The De Giorgi-type iteration methods leading to H\"{o}lder continuity are well known and are available in a wide range of spaces, including homogeneous metric measure spaces. For references we quote \cite{G03,KS01,KMM12,LU68,R18,Z17}.

\section{The case of a general semi-simple, compact, connected Lie group}

The proofs of our results are based on the properties of the commutators listed in Table \ref{tab:hresult} and \eqref{epsiloncom}. This is how we can handle the fact that we don't have a nilpotent stucture.
Similar properties of commutators of vector fields hold in any compact, connected, semi-simple Lie group. For the sake of clarity we presented all details for the case of $\SU (3)$, which is the simplest non-nilpotent group case that takes into account all possible commutators present in the general case.   

Next, we describe those algebraic and analytic properties of semi-simple, compact, connected Lie groups, which allow \textit{mutatis mutandis} for the extension of our proofs in $\SU(3)$ to any semi-simple, compact, connected Lie group.

Let $\G$ be a semi-simple, connected, compact matrix Lie group and $\Lie$ its Lie algebra. Note that every compact Lie group is isomorphic to a compact group of matrices \cite[Corollary 2.40]{HM06}, so there is no loss of generality assuming that $\G$ is a matrix group.

 On $\Lie$ we consider an inner product with properties
$$\langle \Ad g (X) , \Ad g (Y) \rangle = \langle X , Y \rangle , \; \; \text{for all} \; \; g \in \G, \; X, Y \in \Lie \, ,$$
and
$$\langle \ad X (Y) , Z \rangle = - \langle Y , \ad X (Z) \rangle , \; \; \text{for all} \; \; X, Y , Z \in \Lie \, ,$$
where $\Ad g (X) = g X g^{-1}$ and $\ad X (Y) = [X,Y]$. An example of such an inner product is given by any negative multiple of the Killing form \cite{A03}. 

Consider a maximal torus $\To$ of $\G$ and its Lie algebra $\T$, which is a maximal commutative subalgebra of $\Lie$, called a Cartan subalgebra. Let us fix an orthonormal basis ${\mathcal B}_{\T} = \{ T_1, \dots , T_\nu \} $ of $\T$, and identify the dual space $\T^*$ (space of roots)  with $\T$  (space of root vectors).
 
We extend the inner product bi-linearly to the complexified Lie algebra $\Lie_{\comp} = \Lie \oplus i \Lie$.  
The mappings $\ad T \colon \Lie_{\comp} \to \Lie_{\comp}$, $T \in \T$, form a commuting family and are skew-symmetric, so they share eigenspaces and  have purely imaginary eigenvalues.

\begin{definition}\label{def:Root}
We define $R \in \T$ to be a root if $R \neq 0$ and the root space $\Lie_{R} \neq \{0\}$, where
\begin{equation}\notag  \label{eq:RootSpace}
\Lie_{R} = \{ Z \in \Lie_{\comp} \; : \; [T , Z ] = i \; \langle R, T\rangle \, Z  \, , \; \; \text{for all} \; \; T \in \T \;  \}\, .
\end{equation}
\end{definition}

Let $\roots$ be the set of all roots. We call a root positive if its first non-zero coordinate relative to the ordered basis ${\mathcal B}_{\T}$ is positive and let $\proots$ denote the set of all positive roots.

For the following properties of the real root space decomposition we quote \cite[Proposition 6.45, Theorem 6.49]{HM06}. We have 
 $$\Lie = \T \oplus \Ho \, $$ where
 \begin{equation}\label{eq:RootSpaceDecomposition}
\Ho = \T^{\perp} = \bigoplus_{R \in \proots} \, \Ho_R \;  \; \text{and} \; \; 
\Ho_R = (\Lie_{R} \oplus \Lie_{-R} ) \cap \Lie \, .
 \end{equation}
Therefore, we can choose an orthonormal basis of $\Ho $,
\begin{equation}\label{eq:HorizontalBasis}
{\mathcal B}_{\Ho} = \{ X_1, X_2 , \dots , X_{2n -1}, X_{2n} \} \, ,
\end{equation}
with the following properties:
\begin{equation} \label{eq:HorizontalBasisProperties}
 \begin{split}
 (i) & \; \; \text{For all} \; 1\leq j \leq n \; \text{there exists}  \; R_j \in \proots \; \text{such that} \\ & \; \;  \spa \{X_{2j-1}, X_{2j} \} = \Ho_{R_j}.\\
(ii) & \; \; [X_{2j-1}, X_{2j}]= - R_j , \;  [X_{2j}, R_{j}]= - ||R_j||^2 X_{2j-1} , \\ & \hspace{0.25cm}  [R_{j}, X_{2j-1}]= ||R_j||^2 X_{2j} \, .\\
(iii) & \; \; \text{If} \; (m , k) \neq (2j-1, 2j), \; \text{then} \; \; [X_m , X_k ] \in \Ho \, .\\
(iv) & \; \; \text{If} \; T \in \T , \; \text{then} \; \left\{ [X_{2j-1}, T] , \; [X_{2j} , T ] \right\} \subset \Ho_{R_j} \, .  
\end{split}
\end{equation}

Notice that \cite[Proposition 2.20]{A03} the positive roots span the Cartan subalgebra $\T$, but might not form a linearly independent set. To extend ${\mathcal B}_{\Ho}$ to a basis of $\Lie$, let us select a subset of positive roots $\{R_1, ... , R_\nu \}$, which form a basis of $\T$. This can be the set of simple roots, but not necessarily. \par
For $0 < \eps < 1$, define the following vector fields:
\begin{itemize}
	\item For $i \in \{1 ,2n\}$ define $X_i^{\eps} =  X_i$.
	\item For $j \in \{1, \nu\}$ define $R_j^{\eps} = \eps R_j$.
\end{itemize}
Consider the Riemannian approximation given by setting as an orthonormal basis of $\Lie$
the vector fields
$$\{ X_1, ... , X_{2n}, R_1^\eps, ..., R_\nu^\eps \}.$$ 
We now set the \textit{horizontal}  and \textit{vertical} gradients
$$\he u= \sum_{i=1}^{2n} (X_i u) X_i, \hspace{10pt}  \vv u=  \sum_{j=1}^{\nu} (R_j u) R_i, $$ and
the full Riemannian gradient
$$ \nabla^{\eps} u= \he u+ \epsilon \vv u.$$ We also set
$$\we = \delta + |\nabla^{\eps} \ue |^2,$$ and
$$ \ve u = \epsilon \vv u$$

Let us fix a bi-invariant Haar-measure in $\G$.
Consider a domain $\Omega \subset \G$, and the following quasilinear subelliptic equation:
\begin{equation} \label{eq:PDE1G}
\sum_{i=1}^{2n} X_i \left( a_i (\he u) \right) = 0 \, , \;
 \mbox{in} \; \Omega \; ,
\end{equation} 
where for some $0 \leq \delta \leq 1$, $p > 1$ , $0 < l < L$, and for all $\eta, \xi \in {\mathbb R}^{2n}$ the following properties hold:
\begin{align}\label{elliptic11}
\sum_{i,j=1}^{2n} \frac{\partial a_i}{\partial \xi_j}(\xi ) \;
\eta_i \eta_j &
\geq l \Bigl( \delta + |\xi|^2 \Bigr)^{\frac{p-2}{2}} |\eta|^2 \, ,\\\label{elliptic22}
\sum_{i,j=1}^{2n} \left| \frac{\partial a_i}{\partial \xi_j} (\xi
)\right|
& \leq L \Bigl( \delta + |\xi|^2 \Bigr)^{\frac{p-2}{2}} \, ,\\\label{elliptic33}
|a_i (\xi )| & \leq  L \left( \delta + |\xi |^2
\right)^{\frac{p-1}{2}} \, .
\end{align}

We list our main results for a general semi-simple, compact, connected Lie group $\G$.

\begin{theorem} \label{thm:LipschitzG}
Let $p > 1$ and $u \in W_{\Ho,\loc}^{1,p} (\Omega )$ be a weak solution of \eqref{eq:PDE1G}.
Then there exists a constant $c>0$, depending only on $\G, p, l, L$, such that for any Carnot-Carath\`eodory  ball $B_r \subset \subset \Omega$ we have 
\begin{equation}\label{ineq:LinfinityG}
\sup_{B_{r/2}} |\he u | \leq c \left( -\hspace{-0.45cm}\int_{B_r} (\delta + |\he u |^2 )^{\frac{p}{2}} dx \right)^{\frac{1}{p}} \, .
 	\end{equation}

\end{theorem}

\begin{theorem} \label{thm:GradientHolderG}
Let $p \geq 2$ and $u \in W_{\Ho,\loc}^{1,p} (\Omega )$ be a weak solution of \eqref{eq:PDE1G}. Then $\he u \in C^{\alpha}_{\loc} (\Omega)$.
\end{theorem}


Regarding the Riemannian approximation as $\eps \to 0$, by \eqref{eq:HorizontalBasisProperties},  the  commutation relations that arise are exactly the same as  those described in \eqref{epsiloncom}. This means that all the proofs in Sections \S 3 and \S 4 carry over with minor modifications (for example, the homogeneous dimension is $Q= 2n+2\nu$),  and our results are valid in any semi-simple, compact, connected Lie group $\G$. \par 
\bibliographystyle{alpha}
\bibliography{DomokosManfredi}
\end{document}